\documentclass[fleq]{amsart}
\usepackage{amssymb,amsmath,latexsym,amsbsy,color,enumerate}
\numberwithin{equation}{section}

\newtheorem{theorem}{Theorem}

\newtheorem{lemma}[theorem]{Lemma}

\newtheorem{corollary}[theorem]{Corollary}

\begin{document}
 \title{On automorphisms of blowups of projective manifolds}
 \author{Tuyen Trung Truong}
    \address{Department of Mathematics, Syracuse University, Syracuse NY 13244}
 \email{tutruong@syr.edu}
\thanks{}
    \date{\today}
    \begin{abstract}  
In this paper we mainly study the following question: For what projective manifold $X$  of dimension $\geq 3$ that any $f\in Aut(X)$ has zero topological entropy?  Using some non-vanishing conditions on nef cohomology classes, we study the case where $X\rightarrow X_0$ is a finite blowup along smooth centers, here $X_0$ is a projective manifold of interest. Here we allow $X_0$ to be either one of the following manifolds: it has Picard number $1$, or a Fano manifold, or it is a projective hyper-K\"ahler manifold. We also allow the centers of blowups to have large dimensions relative to that of $X_0$ (may be upto $dim(X_0)-2$). Explicit constructions are given in Section \ref{SectionBlowupsAndNonVanishingConditions}, where we also show that the assumptions in the results in that section are necessary (see Example 6 in Section \ref{SectionBlowupsAndNonVanishingConditions}).   

As a consequence, we obtain new examples of manifolds $X$, whose any automorphism is either of zero topological entropy or is cohomologically hyperbolic.  
  
\end{abstract}
\maketitle

\section{Introduction}
\label{Introduction}

The structure of the automorphism group of a compact K\"ahler manifold has been very extensively studied. Among many other things, the following question is very interesting: What compact K\"ahler manifolds have automorphisms of positive entropies? By Gromov-Yomdin's theorem, this question reduces to the one about whether there is an automorphism with first dynamical degree $>1$. We recall that if $X$ is a compact K\"ahler manifold of dimension $k$ and $f:X\rightarrow X$ is a surjective holomorphic map, then the $p$-th dynamical degree $\lambda _p(f)$ of $f$ (here $0\leq p\leq k$) is the spectral radius of the pullback map $f^*:H^{p,p}(X)\rightarrow H^{p,p}(X)$. The theorem of Gromov and Yomdin states that the topological entropy is related to the dynamical degrees as follows
\begin{eqnarray*} 
h_{top}(f)=\max _{p=1,\ldots ,k}\log \lambda _p(f). 
\end{eqnarray*}
Since the dynamical degrees are log-concave (i.e. $\lambda_{p-1}(f)\lambda _{p+1}(f)\leq \lambda _p(f)^2$), we deduce that $h_{top}(f)>0$ if and only if $\lambda _1(f)>1$.

In dimension $k=2$, there are many constructions of automorphisms of positive entropies, starting as early as with the work of Coble (see Dolgachev-Ortland \cite{dolgachev-ortland}) who used Coxeter groups, see works of Cantat \cite{cantat}, Bedford-Kim \cite{bedford-kim2}\cite{bedford-kim1}\cite{bedford-kim}, McMullen \cite{mcmullen3}\cite{mcmullen2}\cite{mcmullen1}\cite{mcmullen}, Oguiso \cite{oguiso3}\cite{oguiso2}, Cantat-Dolgachev \cite{cantat-dolgachev}, Zhang \cite{zhang}, Diller \cite{diller}, D\'eserti-Grivaux \cite{deserti-grivaux}, Reschke \cite{reschke}, Uehara \cite{uehara}... 

In dimension $k\geq 3$, there are several general results on the 
structure of automorphism groups, see Bochner-Montgomery \cite{bochner-montgomery}, Fujiki \cite{fujiki}, Lieberman \cite{lieberman}, Dinh-Sibony \cite{dinh-sibony}, Oguiso \cite{oguiso0}, Keum-Oguiso-Zhang \cite{keum-oguiso-zhang}, Zhang \cite{zhang0},...However, the examples of compact K\"ahler manifolds having automorphisms of positive entropies are very rare (see Oguiso \cite{oguiso1}\cite{oguiso}, Oguiso-Perroni \cite{oguiso-perroni},...). On the other hand, for a class of maps very close to automorphisms, that is the class of pseudo-automorphisms or automorphisms in codimension $1$, there are systematic constructions of many examples of first dynamical degrees greater than $1$ by Bedford-Kim \cite{bedford-kim3}, Perroni-Zhang \cite{perroni-zhang}, Blanc \cite{blanc},... This leads to the natural question: How common are compact K\"ahler manifolds of dimension $\geq 3$ having automorphisms of positive entropies?

In our previous paper \cite{truong} on automorphisms of blowups of $\mathbb{P}^3$, we showed using a heuristic argument that for a "generic" compact K\"ahler manifold $X$, if $f\in Aut(X)$ then $h_{top}(f)=0$. Combined with results of Bayraktar and Cantat (see below), the same argument shows that for a "generic" compact K\"ahler manifold, the automorphism group $Aut(X)$ has only finitely many connected components. The idea is as follows. For an automorphism $f$, there is a non-zero nef cohomological class $\eta \in H^{1,1}(X)$ (we recall that nef classes are in the closure of K\"ahler classes) such that $f^*(\eta )=\lambda _1(f)\eta$. We observe that for a "generic" compact K\"ahler manifold $X$, the non-vanishing condition $B(0,0)$ below is satisfied, hence $\lambda _1(f)=1$ and $h_{top}(f)=0$. Here "generic" is used in the following sense: 

{\bf Expectation of randomness for the intersection ring}. We expect that when we choose randomly a compact K\"ahler manifold $X$ of dimension $k$ with a fixed value of $dim(H^{1,1}(X))$, because $dim (H^{1,1}(X))=dim (H^{k-1,k-1}(X))$ by the Poicare duality, the map 
$$(\zeta _1,\ldots ,\zeta _{k-1})\in H^{1,1}(X)^{k-1}\mapsto \zeta _1.\zeta _2.\ldots .\zeta _{k-1}\in H^{k-1,k-1}(X)$$ 
behaves randomly. In particular, the map $\zeta \in H^{1,1}(X)\mapsto \zeta ^{k-1}\in H^{k-1,k-1}(X)$ should be non-degenerate. 

(This expectation of randomness is related to polar hypersurfaces, see the book Dolgachev \cite{dolgachev} for more detail on polar hypersurfaces.)
 
Note that for hyper-K\"ahler manifolds of dimension $k=2l$, the expectation of randomness with a smaller exponent was proved previously by Verbitsky \cite{verbitsky}: The map $(\zeta _1,\ldots ,\zeta _l)\mapsto \zeta _1.\ldots .\zeta _l$ is non-degenerate. In particular, if $\zeta$ is non-zero then $\zeta ^l$ is non-zero. 

There are of course many manifolds for which this expectation of randomness is not satisfied. For example, if we start from any manifold $X_0$ and let $X\rightarrow X_0$ be a finite blowup of $X_0$ along smooth centers, the resulting manifold $X$ may not satisfy this expectation of randomness. Thus we may hope that $X$ contains some automorphisms of positive entropies, even if $X_0$ does not. A common approach in finding automorphisms of positive entropies, used very efficiently in the case $k=2$, is as follows: Start with a manifold $X_0$ of interest, find a birational meromorphic map $f:X_0\rightarrow X_0$, and then find a finite blowup $X\rightarrow X_0$ such that the lifting of $f$ to $X$ is an automorphism of positive entropy. Hence we can restate the question at the beginning of this section in the following form: Given a manifold of interest $X_0$, is there a finite blowup $X\rightarrow X_0$ carrying an automorphism of positive entropy? In the same paper \cite{truong}, using the non-vanishing condition $A(0,0)$ below, we constructed systematically many finite blowup $X\rightarrow \mathbb{P}^3$ whose any automorphism has zero topological entropy. By the results of Bayraktar and Cantat (see below), it turns out that for many of these examples, the automorphism group has only finitely many connected components. This suggests that the answer to the above question is No for $X_0=\mathbb{P}^3$. 

In recent works, Bayraktar-Cantat \cite{bayraktar}\cite{bayraktar-cantat} used a generalized non-vanishing condition $B(r,0)$ (here $k>2r+2$) of the non-vanishing condition $B(0,0)$ to show that if $X_0$ is a projective manifold of dimension $k\geq 3$ and has Picard number $1$, and $X\rightarrow X_0$ is a finite blowup along smooth centers of dimension $\leq r$ then the automorphism group $Aut(X)$ has only finitely many connected components, in particular if $f\in Aut(X)$ then $h_{top}(f)=0$. In fact (see Theorem \ref{TheoremEk}), for the projective examples given in \cite{bayraktar}\cite{bayraktar-cantat}, a stronger non-vanishing condition, which is very similar to the expectation of randomness for the intersection ring, is satisfied: If $\zeta \in H^{1,1}(X)$ (or $\zeta \in NS_{\mathbb{R}}(X)$ in the case of Picard number $1$) is non-zero ($\zeta$ needs not to be nef) then $\zeta ^{k-r-1}\not= 0$. Previously, for a hyper-K\"ahler manifold $X$ of dimension $k=2l$, Oguiso \cite{oguiso1} used the non-vanishing condition B(l-1,0) of Verbitsky \cite{verbitsky} (note that here the condition $k>2r+2$ is not satisfied) to show that if $f\in Aut(X)$ then either $h_{top}(f)=0$ or it is cohomologically hyperbolic, more precisely the middle dynamical degree $\lambda  _{l}(f)$ is larger than other dynamical degrees.   

In the current paper we combine the ideas in \cite{bayraktar}\cite{bayraktar-cantat},  and \cite{truong} to explore more general situations. More precise, while the results in \cite{bayraktar}\cite{bayraktar-cantat} work in any dimension and yield better conclusion on the structure of the automorphism group, their arguments and conditions apply only for the case $X_0$ has Picard number $1$ and for blowups along centers of dimension $< (k-2)/2$ (e.g. when $k=3$ their arguments apply only for point blowups, while in the paper \cite{truong} we allow blowing up along curves satisfying certain conditions). Using some generalizations of the non-vanishing conditions $B(r,0)$ and $A(0,0)$, in this paper we can work with other manifolds $X_0$ (e.g. Fano manifolds,  projective hyper-K\"ahler manifolds) and allow blowing ups along smooth centers of arbitrary dimensions under certain conditions.      

We now introduce two non-vanishing conditions, including as special cases those referred to in the above. We recall that by Hodge decomposition, on a compact K\"ahler manifold $X$ we have $H^2(X,\mathbb{C})=H^{2,0}(X)\oplus H^{1,1}(X)\oplus H^{0,2}(X)$, and we define $H^{1,1}(X,\mathbb{Z})=H^2(X,\mathbb{Z})\cap H^{1,1}(X)$, $H^{1,1}(X,\mathbb{Q})=H^2(X,\mathbb{Q})\cap H^{1,1}(X)$ and $H^{1,1}(X,\mathbb{R})=H^2(X,\mathbb{R})\cap H^{1,1}(X)$. The cone of nef classes $H_{nef}^{1,1}(X)\subset H^{1,1}(X)$ is the closure of K\"ahler classes.

{\bf Non-vanishing condition $A(r,q)$}: Let $r$, $q$ and $k$ be integers with $-1\leq r\leq k-1$ and $0\leq q\leq k-r-1$. We say that a compact K\"ahler manifold $X$ of dimension $k$ satisfies the non-vanishing condition $A(r,q)$ if for any nef class $\zeta$ on $X$, if $\zeta ^{k-r-1-q}.K_X^{q}=0$ then $\zeta$ is proportional to a rational cohomology class, i.e. a class in $H^{1,1}(X,\mathbb{Q})$. Here $K_X$ is the canonical divisor of $X$.  

{\bf Non-vanishing condition $B(r,q)$}: Let $r$, $q$ and $k$ be integers with $-1\leq r\leq k-1$ and $0\leq q\leq k-r-1$. We say that a compact K\"ahler manifold $X$ of dimension $k$ satisfies the non-vanishing condition $B(r,q)$ if for any non-zero nef class $\zeta$ on $X$ then $\zeta ^{k-r-1-q}.K_X^q\not= 0$. 

These non-vanishing conditions also have algebraic analogs, which we will mainly use in the rest of this paper. We let $Pic(X)$ be the Picard group of $X$, and let $NS(X)=Pic(X)/$(algebraic equivalence), which can be regarded as a subset of $H^2(X,\mathbb{Z})$ via the Chern map $L$ a divisor $\mapsto c_1(L)\in H^{2}(X,\mathbb{Z})$. By Lefschetz $(1,1)$ theorem (see Chapter 0 in the book Griffiths-Harris \cite{griffiths-harris}), we have $NS(X)=H^2(X,\mathbb{Z})\cap H^{1,1}(X)$.  We define $NS_{\mathbb{Q}}(X)=NS(X)\otimes _{\mathbb{Z}}\mathbb{Q}\subset H^{2}(X,\mathbb{Q})$ and $NS_{\mathbb{R}}(X)=NS(X)\otimes _{\mathbb{Z}}\mathbb{R}\subset H^2(X,\mathbb{Z})$. The nef cone $Nef(X)=H_{nef}^{1,1}(X)\cap NS_{\mathbb{R}}(X)\subset NS_{\mathbb{R}}(X)$ is the closure of ample divisors (with real coefficients). In fact, let $u\in Pic_{\mathbb{R}}(X)$ represent a class in $H_{nef}^{1,1}(X)\cap NS_{\mathbb{R}}(X)$.     Theorem 4.5 in Demailly-Paun \cite{demailly-paun} shows that $u$ is then nef in the algebraic geometry sense. Then it follows by Kleiman's result that $u$ is a limit of ample divisors (with real coefficients), see e.g. Corollary 1.4.9 in Lazarsfeld \cite{lazarsfeld}.

{\bf Non-vanishing condition $NA(r,q)$}: Let $r$, $q$ and $k$ be integers with $-1\leq r\leq k-1$ and $0\leq q\leq k-r-1$. We say that a complex projective manifold $X$ of dimension $k$ satisfies the non-vanishing condition $NA(r,q)$ if for any nef class $\zeta\in NS_{\mathbb{R}}(X)$, if $\zeta ^{k-r-1-q}.K_X^{q}=0$ then $\zeta$ is proportional to a rational class, i.e. a class in $NS_{\mathbb{Q}}(X)$. Here $K_X$ is the canonical divisor of $X$.  

{\bf Non-vanishing condition $NB(r,q)$}: Let $r$, $q$ and $k$ be integers with $-1\leq r\leq k-1$ and $0\leq q\leq k-r-1$. We say that a complex projective manifold $X$ of dimension $k$ satisfies the non-vanishing condition $NB(r,q)$ if for any non-zero nef class $\zeta\in NS_{\mathbb{R}}(X)$ then $\zeta ^{k-r-1-q}.K_X^q\not= 0$. 

The use of these non-vanishing conditions to the question on the existence of automorphisms of positive entropies are given in the following two results.   

\begin{theorem}
Let $r$ be an integer such that $k>2r+2$, and let $q$ be such that $k-r-1-q>r+1$. If $X$ satisfies the non-vanishing condition $A(r,q)$ then for any $f\in Aut(X)$ we have $h_{top}(f)=0$. The same result holds if we replace $A(r,q)$ by $NA(r,q)$.
\label{TheoremArqCondition}\end{theorem}

\begin{theorem}
Let $r$ be an integer such that $k>2r+2$, and let $q$ be such that $k-r-1-q>r+1$. If $X$ satisfies the non-vanishing condition $B(r,q)$ then the automorphism group $Aut(X)$ has only finitely many connected components. The same result holds if we replace $B(r,q)$ by $NB(r,q)$.
\label{TheoremBrqCondition}\end{theorem}

We list here some applications of Theorems \ref{TheoremArqCondition} and \ref{TheoremBrqCondition}. The first result concerns blowups of manifolds of Picard number $1$. 

\begin{theorem}
Let $X_0$ be a projective manifold of dimension $k$ and has Picard number $1$ then $X_0$ satisfies the non-vanishing condition $NB(r,0)$ for any $r\geq 0$.

1) Let $V_1,V_2,\ldots ,V_m\subset X_0$ are pairwise disjoint submanifolds, where each $V_j$ is either of dimension $\leq 1$ or is a complete intersection of smooth hypersurfaces of $X_0$. Let $X_1\rightarrow X_0$ be the blowup along $V_1,V_2,\ldots ,V_m$. Then for a generic choice of $V_1,V_2,\ldots ,V_m$ the resulting space $X_1$ satisfies the non-vanishing condition $NB(1,0)$. Here generic is used in the sense of algebraic geometry, i.e. the claim is true out of a proper subvariety of the parameter space.  

2) Let $X_1$ be a blowup satisfying the generic condition in 1). If $\pi :X=X_{n}\rightarrow X_{n-1}\rightarrow \ldots\rightarrow X_1$ a finite blowup along smooth centers of dimensions $<(k-2)/2$, then any automorphism of $X$ has zero topological entropy. If moreover $k\geq 5$ or $X_1=X_0$ then the automorphism group $Aut(X)$ has only finitely many connected components. 

3) In 2) we can also allow the centers of blowups to have larger dimensions, by using Theorem \ref{TheoremEk} and Corollary \ref{CorollaryExtremeCases}, see Examples 1-6 in Section \ref{SectionBlowupsAndNonVanishingConditions}. 
\label{TheoremPicardNumber1}\end{theorem}
{\bf Remarks.} Part 1) is compatible with the expectation of randomness for intersection ring, here the exponent $k-2$ is only $1$ less than the exponent $k-1$ in the expectation of randomness. In part 1), when $3\leq dim(X_0)\leq 4$ we can also show that $X_1$ satisfies the non-vanishing condition $NA(0,0)$ (see Theorem 2 in \cite{truong} for the case $dim(X_0)=3$ and see Lemma \ref{LemmaP4} for the case $dim(X_0)=4$.) The special case of part 2) when $X_1=X_0$ was proved in \cite{bayraktar}\cite{bayraktar-cantat}. 

The next result concerns blowups of Fano manifolds. Recall that a manifold $X$ is Fano if its anti-canonical divisor is ample. 
\begin{theorem}
1) Assume that $X_0$ is a projective manifold of dimension $k$ satisfying the non-vanishing condition $NB(l,0)$ and is Fano, i.e. $-K_X$ is ample. If $0\leq k-l-3$ and $0\leq 2l+3-k$, and $X\rightarrow X_0$ is a finite composition of blowups along smooth centers of dimensions $\leq k-3-l$, then the automorphism group $Aut(X)$ has only finitely many connected components.

Similar results hold if we replace $NB(l,0)$ by $A(l,0)$, $NA(l,0)$ and $B(l,0)$.

2) Let $X_0=\mathbb{P}^{k_1}\times \mathbb{P}^{k_2}\times \ldots \times\mathbb{P}^{k_m}$ be a multi-projective space with $m\geq 2$, and we arrange that $1\leq k_1\leq k_2\leq \ldots k_m$. Denote by $k=k_1+\ldots +k_m$ the dimension of $X_0$. 

i) If $2\leq l\leq k_1+k_2$ and $2l\leq k+1$, and $X\rightarrow X_0$ is a finite blowup along smooth centers of dimensions $\leq l-2$ then any automorphism $f\in Aut(X)$ has zero topological entropy. 

ii) If $2\leq k_1$, and $X\rightarrow X_0$ is a finite blowup along smooth centers of dimensions $\leq k_1-2$ then the automorphism group $Aut(X)$ has only finitely many connected components.     

3) In 1) and 2) we can also allow the centers of blowups to have larger dimensions, by using Theorem \ref{TheoremEk} and Corollary \ref{CorollaryExtremeCases}, see Examples 1-6 in Section \ref{SectionBlowupsAndNonVanishingConditions}. 
\label{TheoremAmpleCanonicalDivisor}\end{theorem}
For example, if $X_0=\mathbb{P}^{k_1}\times \mathbb{P}^{k_2}\times \ldots \times\mathbb{P}^{k_m}$ is a multi-projective space, where $m\geq 2$ and $dim(X_0)\geq 3$, and $X\rightarrow X_0$ is a finite composition of point-blowups then the automorphism group $Aut(X)$ has only finitely many connected components. Hence for any $f\in Aut(X)$ we have $h_{top}(f)=0$. In particular, the pseudo-automorphisms constructed in \cite{perroni-zhang} can never be automorphisms. 

Besides helping to check whether a given manifold can only have automorphisms of zero topological entropy, the above non-vanishing conditions also help in checking whether a given manifold must have interesting automorphisms in case it has automorphisms of positive entropies, or to give constraints to dynamical degrees of holomorphic maps of a given manifold in general. We illustrate this with a result concerning automorphisms which are cohomologically hyperbolic. We recall that a surjective holomorphic map $f:X\rightarrow X$ is cohomologically hyperbolic if it has one dynamical degree larger than other dynamical degrees. In particular, such an automorphism has positive topological entropy.

\begin{theorem}
Let $X_0$ be a projective hyper-K\"ahler manifold of dimension $k=2l$, and let $X\rightarrow X_0$ be a finite composition of blowups along smooth centers of dimension $\leq l-1$. Then any automorphism $f\in Aut(X)$ is either of zero topological entropy or is cohomologically hyperbolic with the dominant dynamical degree $\lambda _l(f)$.

Moreover, for any automorphism $f\in Aut(X)$ and $0\leq p\leq l$ we have $\lambda _{2l-p}(f)=\lambda _p(f)=\lambda _1(f)^p$. 

We can also allow the centers of blowups to have larger dimensions by using Theorem \ref{TheoremEk} and Corollary \ref{CorollaryExtremeCases}, see Examples 1-6 in Section \ref{SectionBlowupsAndNonVanishingConditions}.   

In the above we can also start from other manifolds of even dimensions, such as $X_0$ has Picard number $1$.
\label{TheoremHyperKahler}\end{theorem} 
Note that the case where $X=X_0=$ a compact hyper-K\"ahler manifold was proved in \cite{oguiso1}, where some specific examples were also given. Cohomologically hyperbolic automorphisms have been shown to have good dynamical properties, see the papers Cantat \cite{cantat}, Dinh-Sibony \cite{dinh-sibony3}, \cite{dinh-sibony4}, \cite{dinh-sibony5} and Dinh-deThelin \cite{dinh-deThelin} for more details. 

This paper is arranged as follows. In Section \ref{SectionBlowupsAndNonVanishingConditions}, we explore the question when a blowup preserves the non-vanishing conditions $A(r,q)$, $B(r,q)$, $NA(r,q)$ and $NB(r,q)$, and give many explicit examples at the end of the section. The proofs of Theorems \ref{TheoremArqCondition}, \ref{TheoremBrqCondition}, \ref{TheoremPicardNumber1}, \ref{TheoremAmpleCanonicalDivisor} and \ref{TheoremHyperKahler} are given in Section \ref{SectionProofsOfTheorems}. 

{\bf Acknowledgments.} We thank Keiji Oguiso for suggesting extending the results in dimension $3$ to higher dimensions and for helpful correspondences. We thank Igor Dolgachev for explaining several aspects of automorphism groups, and thank Mattias Jonsson and Turgay Bayraktar for helpful comments and for pointing out the reference \cite{paun} which helped to improve the paper. 

\section{Blowups and the non-vanishing conditions $A(r,q)$, $NA(r,q)$, $B(r,q)$ and $NB(r,q)$}
\label{SectionBlowupsAndNonVanishingConditions}

In this section we explore the question: if $Y$ is a projective manifold satisfying the non-vanishing condition $A(r,q)$ (correspondingly the non-vanishing conditions $NA(r,q)$, $B(r,q)$ and $NB(r,q)$), and $\pi :X\rightarrow Y$ is a blowup along a smooth submanifold $V\subset Y$, does $X$ also satisfy the non-vanishing condition $A(r,q)$ (respectively the non-vanishing conditions $NA(r,q)$, $B(r,q)$ and $NB(r,q)$)? We construct many examples of blowups when the answer to the question is Yes. The main idea is that if the normal vector bundle of $V$ in $Y$ is "{\bf negative enough}" and $V$ is "movable" (for the precise conditions see Theorem \ref{TheoremEk}, Lemmas \ref{LemmaP4} and \ref{LemmaManifoldDimension4}, and Corollary \ref{CorollaryExtremeCases}) then $X$ will also satisfy the non-vanishing conditions. Example 6 at the end of the section show that these assumptions of "negative normal vector bundle" and "movability of the center of blowup" $V$ can not be removed. To motivate the constructions, in the first two results we consider blowups of projective manifolds of dimension $4$, and after that will consider blowups of manifolds of arbitrary dimensions. At the end of the section we will give some explicit examples.  

For simplicity, for the results in this section we present the proofs only for the non-vanishing conditions $A(r,q)$ and $B(r,q)$. The proofs for algebraic non-vanishing conditions $NA(r,q)$ and $NB(r,q)$ are similar, using the following result: If $\pi :X\rightarrow Y$ is a birational morphism, then $\pi _*(NS(X))\subset NS(Y)$ and $\pi ^*(NS(Y))\subset NS(X)$ (see Example 19.1.6 in \cite{fulton}). Alternatively, we can see this by using $NS(X)=H^{2}(X,\mathbb{Z})\cap H^{1,1}(X)$ and $NS(Y)=H^{2}(Y,\mathbb{Z})\cap H^{1,1}(Y)$.

The first result of this section is for blowups of $\mathbb{P}^4$. See Theorem \ref{TheoremPicardNumber1} 1) for an extension to higher dimensions. 
\begin{lemma}
Let $V_1,\ldots ,V_n\subset \mathbb{P}^4$ be irreducible pairwise disjoint smooth compact complex submanifolds of dimension $\leq 2$. Let $\pi :X\rightarrow \mathbb{P}^4$ be the blowup of  $\mathbb{P}^4$ at $V_1,\ldots ,V_n$. If for any $j$ one of the following conditions are satisfied, then $X$ satisfies the non-vanishing condition $A(0,0)$

i) $dim(V_j)\leq 1$.

ii) $dim(V_j)=2$, and $V_j$ is a complete intersection of two smooth hypersurfaces $D_1$ and $D_2$ of degrees $d_1$ and $d_2$. For example, we may choose $V_j$ to be the intersection between a smooth hypersurface and a generic hyperplane. By Bertini's theorem, such a $V_j$ is smooth.  

(Note that since the manifolds $V_1,\ldots ,V_n$ are pairwise disjoint, there is at most one of them of dimension $2$.)
\label{LemmaP4}\end{lemma} 
\begin{proof}
Let $H\subset \mathbb{P}^4$ be a hyperplane and let $E_1,\ldots ,E_n$ be the exceptional divisors. 

Let $\zeta$ be a nef class on $X$ with $\zeta ^3=0$, we will show that $\zeta$ is proportional to a rational cohomology class. We can write
\begin{eqnarray*} 
\zeta =aH+\sum _{j=1}^nb_jE_j,
\end{eqnarray*}
for real numbers $a$ and $b_i$. Since $\zeta$ is nef, either $a>0$ or $\zeta =0$. If $\zeta =0$ then we are done. Hence we now assume that $a\not= 0$, and upon dividing by $a$ can write
\begin{eqnarray*}
\zeta =H+\sum _{j=1}^nb_jE_j
\end{eqnarray*}
where $b_j\leq 0$. The conclusion of the lemma is equivalent to that $b_1,b_2,\ldots ,b_n$ are rational numbers. 

By assumption we have $\zeta ^3.E_j=0$ for any $j=1,\ldots ,n$. Since $V_1,\ldots ,V_n$ are disjoint, it follows that $E_i.E_j=0$ for $i\not= j$. Hence $\zeta ^3.E_j=(H+b_jE_j)^3.E_j$ for any $j$. 

Let $\mathcal{E}_j=N_{V_j/\mathbb{P}^4}$ be the normal vector bundle of $V_j$ in $\mathbb{P}^4$. Then $E_j=\mathbb{P}(\mathcal{E}_j)$. Let $c_1(\mathcal{E}_j)$ and $c_2(\mathcal{E}_j)$ be the first and second Chern classes of $\mathcal{E}_j$ (the higher Chern classes vanish since $\mathcal{E}_j$ is the vector bundle over a manifold of dimension $\leq 2$), and let $\pi _j:E_j\rightarrow V_j$ be the projections. Let $h_j=\pi ^*(H)|_{E_j}=\pi _j^*(H|_{V_j})$, and $e_j=E_j|_{E_j}$. Since $h_j^3=\pi _j^*(H^3|_{E_j})=0$, the equality $(H+b_jE_j)^3.E_j=0$ becomes
\begin{equation}
b_j^3e_j^3+3b_j^2h_j.e_j^2+3b_jh_j^2.e_j=0,
\label{Equation1}\end{equation}
in $H^*(E_j)$.

We consider three cases:

Case 1: $V_j=$ a point. In this case $\mathcal{E}_j$ has rank $4$, $c_1(\mathcal{E}_j)=0$ and $c_2(\mathcal{E}_j)$, and we know that (see Remark 3.2.4 in Fulton's book \cite{fulton})
\begin{eqnarray*} 
e_j^4=0.
\end{eqnarray*}
Moreover, we know that $H^*(E_j)$ is generated by $e_j$ as an algebra over $H^*(V_j)$, with the defining relation $e_j^4=0$. Therefore from Equation (\ref{Equation1}) we see that $b_j=0$ hence is a rational number. 

Case 2: $V_j=$ a curve. In this case $\mathcal{E}_j$ has rank $3$,  and $c_2(\mathcal{E}_j)$, and we know that (see Remark 3.2.4 in Fulton's book \cite{fulton})
\begin{eqnarray*} 
e_j^3-\pi _j^*c_1(\mathcal{E}_j)e_j^2=0.
\end{eqnarray*}
Moreover, we know that $H^*(E_j)$ is generated by $e_j$ as an algebra over $H^*(V_j)$, with the defining relation $e_j^3-\pi _j^*c_1(\mathcal{E}_j)e_j^2=0$. If $b_j=0$ then it is a rational number and we are done. If $b_j\not= 0$, dividing Equation (\ref{Equation1}) by $b_j^3$ and compare to the defining equation of $H^*(E_j)$, we find
\begin{eqnarray*}
\frac{3}{b_j}h_j=-\pi _j^*c_1(\mathcal{E}_j).    
\end{eqnarray*}
Since $h_j=\pi _j^*(H|_{V_j})=\pi _j^*(deg(V_j))$ and $\pi _j^*:H^*(V_j)\rightarrow H^*(E_j)$ is injective, we have
\begin{eqnarray*}
\frac{3}{b_j}deg(V_j)=-c_1(\mathcal{E}_j).
\end{eqnarray*}
Because $deg(V_j)$ is a positive integer, it follows that the integer $c_1(\mathcal{E}_j)$ is non-zero, and hence 
\begin{eqnarray*} 
b_j=-\frac{3deg(V_j)}{c_1(\mathcal{E}_j)}\in \mathbb{Q}
\end{eqnarray*}
as wanted. 

Case 3: $V_j=$ a surface, and is a complete intersection of two hypersurfaces of degrees $d_1$ and $d_2$. In this case $\mathcal{E}_j$ has rank $2$, and we know that (see Remark 3.2.4 in Fulton's book \cite{fulton})
\begin{eqnarray*} 
e_j^2-\pi _j^*c_1(\mathcal{E}_j)e_j+\pi _j^*c_2(\mathcal{E}_j)=0.
\end{eqnarray*}
Moreover, we know that $H^*(E_j)$ is generated by $e_j$ as an algebra over $H^*(V_j)$, with the defining relation $e_j^2-\pi _j^*c_1(\mathcal{E}_j)e_j+\pi _j^*c_2(\mathcal{E}_j)=0$.  If $b_j=0$ then it is rational and we are done. Hence we can assume that $b_j\not= 0$. Dividing Equation \ref{Equation1} by $b_j^3$ and defining $a_j=1/b_j$, we find $e_j^3+3a_jh_j.e_j^2+3a_j^2h_j^2.e_j=0$.  Because
\begin{eqnarray*}
&&e_j^3+3a_jh_j.e_j^2+3a_j^2h_j^2.e_j\\
&=&[e_j^2-\pi _j^*c_1(\mathcal{E}_j)e_j+\pi _j^*c_2(\mathcal{E}_j)]\times [e_j+3a_jh_j+\pi _j^*c_1(\mathcal{E}_j)e_j]\\
&&+[3a_j^2h_j^2+3a_jh_j\pi _j^*(c_1(\mathcal{E}_j))+\pi _j^*(c_1(\mathcal{E}_j)^2)-\pi _j^*(c_2(\mathcal{E}_j))]\times e_j,
\end{eqnarray*}
we deduce
\begin{eqnarray*}
3a_j^2h_j^2+3a_jh_j\pi _j^*(c_1(\mathcal{E}_j))+\pi _j^*(c_1(\mathcal{E}_j)^2)-\pi _j^*(c_2(\mathcal{E}_j))=0.
\end{eqnarray*}
Because $\pi _j^*:H^j(V_j)\rightarrow H^*(E_j)$ is injective, we obtain
\begin{equation}
3a_j^2(H^2|_{V_j})+3a_j(H|_{V_j})c_1(\mathcal{E}_j)+c_1(\mathcal{E}_j)^2-c_2(\mathcal{E}_j)=0.
\label{Equation2}\end{equation}
We now use the explicit values of the Chern classes. Example 3.2.12 in \cite{fulton} gives that $c(\mathcal{E}_j)=(1+d_1H|_{V_j})(1+d_2H|_{V_j})$, hence $c_1(\mathcal{E}_j)=(d_1+d_2)H|_{V_j}$ and $c_2(\mathcal{E}_j)=d_1d_2(H^2|_{V_j})$. Therefore $c_1(\mathcal{E}_j)^2=(d_1+d_2)^2(H^2|_{V_j})$. From this, upon dividing Equation (\ref{Equation2}) by the integer number $(H^2|_{V_j})=deg(V_j)>0$, we obtain the equation
\begin{eqnarray*}
3a_j^2+3(d_1+d_2)a_j+(d_1+d_2)^2-d_1d_2=0.
\end{eqnarray*}
Using the quadratic formula, we see that the above has real solutions if and only if $d_1=d_2=d$ a positive integer, and in that case it has a repeated root $a_j=-d$ and hence $b_j=-1/d$ is a rational number, as wanted.  
\end{proof} 
 
The next result considers blowups of general projective manifolds of dimension $4$.
\begin{lemma}
Let $Y$ be a projective manifold of dimension $4$, and $V\subset Y$ be an irreducible compact complex submanifold. Let $\pi :X\rightarrow Y$ be the blowup of $Y$ at $V$, and let $E$ be the exceptional divisor. Assume that $Y$ satisfies the non-vanishing condition $A(0,0)$ (correspondingly the non-vanishing conditions   $NA(0,0)$, $B(0,0)$ and $NB(0,0)$). If either one of the following three conditions is satisfied, then $X$ also satisfies the non-vanishing condition $A(0,0)$ (respectively the non-vanishing conditions $NA(0,0)$, $B(0,0)$ and $NB(0,0)$). 

(i) $V$ is a point.

(ii) $V$ is a curve. In this case let $\mathcal{E}=N_{V/Y}$ be the normal vector bundle of $V$ in $Y$ and $c_1(\mathcal{E})$ be the first Chern class of $\mathcal{E}$. We then assume that $c_1(\mathcal{E})<0$ and $V$ is not the only effective cycle (with real coefficient) in its cohomology class.

(iii) $V$ is a surface. In this case we assume that $\pi _*(E^3)$ can be represented by an effective curve (with real coefficients) whose intersection with $V$ has dimension $\leq 0$, $V$ can be represented by an effective cycle (with real coefficients) intersecting properly with $V$, and $E^4<0$. 
  
Note that the conditions i), ii) and iii) can be stated in a uniform manner, see the proof of this lemma and see also Theorem \ref{TheoremEk}.    
\label{LemmaManifoldDimension4}\end{lemma} 
\begin{proof} (See also the proof of Theorem 2 in \cite{truong} for blowups of $\mathbb{P}^3$).

(iii) Assume that the condition (iii) is satisfied and $Y$ satisfies the non-vanishing condition $A(0,0)$ (the case when $Y$ satisfies the non-vanishing condition $B(0,0)$ is similar; the algebraic analogs $NA(0,0)$ and $NB(0,0)$ are also similar by observing that Neron-Severi groups are preserved by pushing forward by $\pi$). Let $\zeta $ be a nef class on $X$ such that $\zeta ^3=0$. We then show that $\zeta$ is proportional to  a rational cohomology class. We can write $\zeta =\pi ^*(\xi )-aE$ for some $a\geq 0$ and $\xi =\pi _*(\zeta )\in H^{1,1}(Y)$. Then 
\begin{equation}
0=\zeta ^3=\pi ^*(\xi ^3)-3a\pi ^*(\xi ^2).E+3a^2\pi ^*(\xi ).E^2-a^3E^3. 
\label{Equation3}\end{equation} 
Intersecting Equation (\ref{Equation3}) with $E$, and using $\pi _*(E)=0$ and $\pi _*(E.E)=-V$ (see Section 4.3 in \cite{fulton}) we find that
\begin{equation}
a^3E^4=-3a\pi ^*(\xi ^2).E^2+3a^2\pi ^*(\xi ).E^3.
\label{Equation4}\end{equation}
The assumptions imply that $-\pi ^*(\xi ^2).E^2,\pi ^*(\xi ).E^3\geq 0$ are psef. In fact, approximating $\zeta$ by K\"ahler classes, we may assume without loss of generality that $\zeta$ is represented by a positive closed smooth $(1,1)$ form. Then $\xi =\pi _*(\zeta )$ can be represented by a positive closed $(1,1)$ form smooth out of $V$. Since $codim(V)=2$, it follows (see e.g. Section 4 Chapter 3 in the book Demailly \cite{demailly}) that $\xi .\xi $ is represented by a positive closed $(2,2)$ form smooth out of $V$. Then $-\pi ^*(\xi .\xi ).E.E=\xi .\xi .V$ must be non-negative, since by assumptions we can find an effective cycle $V'$ of dimension $2$ in the cohomology class of $V$ and intersect properly with $V$. Then by the results in \cite{demailly} again, we have
\begin{eqnarray*}
 \xi .\xi .V=\xi .\xi .V'
\end{eqnarray*} 
can be represented by a non-negative measure, therefore is non-negative as wanted. Similarly we have $\pi ^*(\xi ).E^3\geq 0$. 

From the above we must have $a=0$. Otherwise, $a>0$ and we obtain a contradiction
\begin{eqnarray*}
0>a^3E^4=-3a\pi ^*(\xi ^2).E^2+3a^2\pi ^*(\xi ).E^3\geq 0.
\end{eqnarray*}

Thus $\zeta =\pi ^*(\xi )$, and by a result of Paun (see \cite{paun})  $\xi$ is itself nef. Then since $\zeta ^3=0$, it follows that $\xi ^3=0$. Because $Y$ satisfies non-vanishing condition A and $\xi$ is nef on $Y$ as shown above, if follows that $\xi $ is proportional to a rational cohomology class and so is $\zeta =\pi ^*(\xi )$.

ii) Assume that the assumption ii) is satisfied. We first observe that the condition $c_1(\mathcal{E})<0$ is equivalent to $E^4<0$. In fact, if $e=E|_E$ and $\pi _E:E\rightarrow V$ is the induced map, we know that
\begin{eqnarray*} 
E^4=e^3=\pi _E^*(c_1(\mathcal{E}))e^2.
\end{eqnarray*}
Now $\pi _E^*(c_1(\mathcal{E}))=c_1(\mathcal{E})\mathbb{P}^2$, where $\mathbb{P}^2$ is a fiber of the map $\pi _E$. Since $E|_E$ is the tautological bundle, it follows that $e.\mathbb{P}^2=-\mathbb{P}^1$ and $e.e.\mathbb{P}^2=-e.\mathbb{P}^1=1$. Therefore $E^4=c_1(\mathcal{E})<0$. 

Thus the assumptions in ii) can be restated as follows: Here $3=4-1$ is the codimension of $V$ in $Y$, $(-1)^{3-1}\pi _*(E^3)=V$ can be represented by an effective cycle (with real coefficients) whose intersection with $V$ has dimension $\leq 0$, and $(-1)^{4-1}\pi _*(E^4)$ is a positive number. Stated this way, we can see that the statements of ii) and iii) are similar. Given this, the proof of ii) is similar to that of iii), and hence is omitted.  

i) The proof is similar to those of iii) and ii) above. 

\end{proof} 
 
The next result concerns blowups of projective manifolds of arbitrary dimension. Part (i) of Theorem \ref{TheoremEk} below refines the results in the papers \cite{bayraktar} and \cite{bayraktar-cantat} (see the remark right after the statement of the theorem). Examples satisfying parts ii), ii') and iii') of Theorem \ref{TheoremEk} are given at the end of this section. In Example 6 at the end of this section, we show that the assumptions in Theorem \ref{TheoremEk} ii') (and those of Lemmas \ref{LemmaP4} and \ref{LemmaManifoldDimension4} and Corollary \ref{CorollaryExtremeCases}) can not be removed.   

\begin{theorem}
Let $Y$ be a projective manifold of dimension $k$, and let $V\subset Y$ be a compact complex submanifold. Let $\pi :X\rightarrow Y$ be the blowup of $Y$ at $V$, and let $E$ be the exceptional divisor. Let $r$ and $q$ be integers with $-1\leq r\leq k-1$ and $0\leq q\leq k-r-1$. Assume that $Y$ satisfies the non-vanishing condition $A(r,q)$ (correspondingly the non-vanishing conditions $NA(r,q)$, $B(r,q)$ and $NB(r,q)$). If one of the following four conditions is satisfied, then $X$ also satisfies the non-vanishing condition $A(r,q)$ (respectively the non-vanishing condition $NA(r,q)$, $B(r,q)$ and $NB(r,q)$).

i) We assume that the dimension of $V$ is $dim(V)\leq r$.  
 
ii) Assume that $q=0$ and the dimension of $V$ is $>r$. Let $s=$codimension of $V$ in $Y$, then $s\leq k-r-1$. We assume that for any $j=s,s+1,\ldots ,k-r-1$, the cycle $(-1)^{j-1}\pi _*(E^j)$ can be represented by an effective cycle (with real coefficients) whose intersect with $V$ has dimension $\leq k-j-1$; and the cycle $(-1)^{k-r-1}\pi _*(E^{k-r})$ is strictly effective, i.e. it is effective and non-zero. Note that the cycles $(-1)^{j-1}\pi _*(E^j)$ can be represented in terms of the Chern classes of $\mathcal{E}$, see the Remark after the proof of the theorem.  

When the center of blowup $V$ has Picard number $1$ (e.g. a curve or a projective space), the assumptions in ii) and iii) can be less restrictive, in that we do not require all the effective cycles $(-1)^{j-1}\pi _*(E^j)$ to have intersections of small enough dimensions with $V$.   

ii') Assume that $q=0$, the dimension of $V$ is $>r$, and $H^{1,1}(V)$ has dimension $1$ (or $V$ has Picard number $1$ for the algebraic non-vanishing conditions $NA(r,0)$ and $NB(r,0)$). Let $s=$codimension of $V$ in $Y$, then $s\leq k-r-1$. We assume that for any $j=s,s+1,\ldots ,k-r-1$, the cycle $(-1)^{j-1}\pi _*(E^j)$ is effective; and the cycle $(-1)^{k-r-1}\pi _*(E^{k-r})$ is strictly effective, i.e. it is effective and non-zero. We also assume that $V$ is not the only effective cycle in its cohomology class. 

iii') Assume that $q>0$, the dimension of $V$ is $>r$, and $H^{1,1}(V)$ has dimension $1$ (or $V$ has Picard number $1$ for the algebraic non-vanishing conditions $NA(r,q) $ and $NB(r,q)$). We assume that the other requirements in ii') are satisfied. In addition we assume that $c_1(Y)|_V$ is ample in $V$.  

\label{TheoremEk}\end{theorem}
Remark: For the proof of i) when $q=0$ and $X\rightarrow X_0$ is a finite blowup along smooth centers of dimension $\leq r$ where $X_0$ has Picard number $1$, see also \cite{bayraktar} and \cite{bayraktar-cantat}.  Our argument here is different from that used in those papers; in particular, here we can show that for the examples given in their papers, if $\zeta \in NS_{\mathbb{R}}(X)$ (not necessarily nef) is such that $\zeta ^{k-r-1}=0$ then $\zeta =0$. 
\begin{proof} 
i)  Assume that $Y$ satisfies the non-vanishing condition $B(r,q)$ (the case of non-vanishing condition $A(r,q)$ is similar; the algebraic analogs $NA(r,q)$ and $NB(r,q)$ are also similar, see the remark at the beginning of this section). Let $\zeta$ be a nef class on $X$ such that $\zeta ^{k-r-1-q}.K_X^q=0$. We need to show that $\zeta =0$. We can write $\zeta =\pi ^*(\xi )-aE$ where $a\geq 0$, and $\xi =\pi _*(\zeta )$. Let $t=dim(V)$ then we can write $K_X=\pi ^*(K_Y)+(k-1-t)E$. We first show that $a=0$. (Note that in the proof of i) we do not need that $a\geq 0$.) 

By assumption we have
\begin{eqnarray*}
0&=&\zeta ^{k-r-1-q}.K_X^q\\
&=&(\pi ^*(\xi )-aE)^{k-r-1-q}.(\pi ^*(K_Y)+(k-1-t)E)^q\\
&=&[\sum _{j=0}^{k-r-1-q}(-a)^{k-r-1-j-q}C(j,k-r-1-q)\pi ^*(\xi ^j).E^{k-r-1-j-q}]\\
&&\times [\sum _{i=0}^q(k-1-t)^{q-i}C(i,q)\pi ^*(K_Y)^iE^{q-i}]\\
&=&\sum _{i,j}(-a)^{k-r-1-j-q}(k-1-t)^{q-i}C(j,k-r-1)C(i,q)\pi ^*(\xi ^j).\pi ^*(K_Y^i).E^{k-r-1-j-i}.
\end{eqnarray*}
Here $C(j,k-r-1-q)$ and $C(i,q)$ are the binomial numbers. Intersecting the above with $E^{r-t+1}$ and then pushing forward by the map $\pi$ we obtain by the projection formula
\begin{eqnarray*}
0=\sum _{i,j}(-a)^{k-r-1-j-q}(k-1-t)^{q-i}C(j,k-r-1)C(i,q)\xi ^j.K_Y^i.\pi _*(E^{k-t-j-i}).
\end{eqnarray*}
Observe that except for the first term $(-a)^{k-r-1-q}(k-1-t)^{q}\pi _*(E^{k-t})$, other terms are zero (in fact, if $i+j>0$ then $\pi _*(E^{k-t-j-i})$ must be zero because it has dimension $t+i+j>t$ and has support in $V=\pi (E)$ which is of dimension $t$); and the first term is $(-1)^{r-t+q}a^{k-r-1-q}(k-1-t)^qV$ (see the formula in Section 4.3 of \cite{fulton}). Therefore $a=0$ as wanted.

Thus $\zeta =\pi ^*(\xi )$, and from \cite{paun}, it follows that $\xi $ is nef. Pushing forward the equality $\pi ^*(\xi ^{k-r-1-q}).(\pi ^*(K_Y)+(k-t-1)E)^q$ by the map $\pi$, from the assumption $k-q\geq r+1>r=dim(V)$ we see as in the above paragraph that $\xi ^{k-r-1-q} .K_Y^q=0$. Then the assumption on $Y$ implies that $\xi =0$ and hence $\zeta =\pi ^*(\xi )=0$.  

ii)  Assume that $Y$ satisfies the non-vanishing condition $B(r,0)$ (the case of non-vanishing conditions $A(r,0)$, $NA(r,0)$ and $NB(r,0)$ are similar). Let $\zeta $ be a nef class on $X$ such that $\zeta ^{k-r-1}=0$. We will show that $\zeta =0$. We can write $\zeta =\pi ^*(\xi )-aE$ where $a\geq 0$, and $\xi =\pi _*(\xi )$. As in the proof of i), it suffices to show that $a=0$. 

The proof proceeds similarly to that of i) and of Lemma \ref{LemmaManifoldDimension4} iii). We have
\begin{eqnarray*}
0=\zeta ^{k-r-1}=\sum _{j=0}^{k-r-1}(-a)^{k-r-1-j}C(j,k-r-1)\pi ^*(\xi ^j).E^{k-r-1-j}. 
\end{eqnarray*}
Intersecting with $E$ we obtain
\begin{eqnarray*}
0=\sum _{j=0}^{k-r-1}(-a)^{k-r-1-j}C(j,k-r-1)\pi ^*(\xi ^j).E^{k-r-j}. 
\end{eqnarray*}
Pushing forward the above equality by $\pi$, we obtain 
\begin{equation}
0=\sum _{j=0}^{k-r-1}(-a)^{k-r-1-j}\xi ^j.\pi _*(E^{k-r-j}). 
\label{Equation5}\end{equation}
In Equation (\ref{Equation5}), the terms corresponding with $k-r-j<s$ (or equivalently $j>k-r-s$) are zero. Hence
\begin{eqnarray*} 
0=\sum _{j=0}^{k-r-s}(-a)^{k-r-1-j}\xi ^j.\pi _*(E^{k-r-j}). 
\end{eqnarray*} 
By assumption and the argument in the proof of Lemma \ref{LemmaManifoldDimension4} iii), each individual term in the above is effective, and in the first term $(-1)^{k-r-1}\pi _*(E^{k-r})$ is strictly effective. Since $a\geq 0$, this implies that $a=0$ as wanted.  

ii') We follow the proof of ii). Let $\omega$ be an ample class on $Y$. Since $V$ is not the only effective curve in its cohomology class, it follows that $\xi .V$ is effective. Let $\iota :V\subset Y$ be the inclusion. Because $V$ has Picard number $1$, it follows that $\xi |_V=\iota ^*(\xi )=b\omega |_{V}$ for some real number $b$. Since $b\omega .V=\iota _*(\xi |_V)=\xi .V$, and $\xi .V$ is effective and $\omega .V$ is strictly effective, it follows that $b\geq 0$. The equation $\zeta ^{k-r-1}=0$ implies
\begin{eqnarray*}
\zeta ^{k-r-1}.\pi ^*(\omega ^r).E=0,
\end{eqnarray*}
and the latter is the same as 
\begin{eqnarray*}
(b\pi ^*(\omega )-aE)^{k-r-1}.\pi ^*(\omega ^r).E=0.
\end{eqnarray*}
Then we can proceed as in the proof of ii).

iii') Assume that $Y$ satisfies the non-vanishing condition $B(r,q)$ (the case of non-vanishing conditions $A(r,q)$, $NA(r,q)$ and $NB(r,q)$ are similar). Let $\zeta $ be a nef class on $X$ such that $\zeta ^{k-r-1-q}.K_X^q=0$. We will show that $\zeta =0$. We can write $\zeta =\pi ^*(\xi )-aE$ where $a\geq 0$, and $\xi =\pi _*(\xi )$. First, we show that $a=0$.

Note that $\zeta ^{k-r-1-q}.K_X^q=0$ implies $\zeta ^{k-r-1-q}.K_X^q.\pi ^*(K_Y)^r.E=0$, and the latter is the same as $\zeta ^{k-r-1-q}.c_1(X)^q.\pi ^*(c_1(Y))^r.E=0$. The latter is the same as $\zeta ^{k-r-1-q}|_E.c_1(X)^q_E.\pi ^*(c_1(Y))^r|_E=0$. Since $V$ has Picard number $1$ and $c_1(Y)|_V$ is effective, as in the proof of ii') we can write $c_1(Y)|_V=T|_V$, where $T$ is either zero or ample on $Y$. We have 
$$\pi ^*(c_1(Y))|_E=\pi _E^*(c_1(Y)|_V)=\pi _E^*(T |_V)=\pi ^*(T )|_E.$$ 
Using $c_1(X)=\pi ^*(c_1(Y))-aE$, we can then write
\begin{eqnarray*}
c_1(X)|_E=(\pi ^*(T )-(s-1)E)|_E.
\end{eqnarray*}
Let $\omega$ be an ample divisor on $Y$. Then the original equation $\zeta ^{k-r-1-q}.c_1(X)^q.\pi ^*(c_1(Y))^r.E=0$ becomes 
\begin{eqnarray*}
(\pi ^*(\xi )-aE)^{k-r-1-q}(\pi ^*(T)-(s-1)E)^q.E.\pi ^*(\omega )^r=0. 
\end{eqnarray*}
Pushforward this equation by $\pi$, we find by the projection formula 
\begin{eqnarray*}
\pi _*[(\pi ^*(\xi )-aE)^{k-r-1-q}(\pi ^*(T )-(s-1)E)^q.E].\omega ^r=0.
\end{eqnarray*}
Arguing as in the proof of ii) we see that if $a>0$ then the term $\pi _*[(\pi ^*(\xi )-aE)^{k-r-1-q}(\pi ^*(T)-(s-1)E)^q.E]$ is psef and non-zero, and hence  $\pi _*[(\pi ^*(\xi )-aE)^{k-r-1-q}(\pi ^*(T)-(s-1)E)^q.E].\omega ^r$ can not be zero since $\omega$ is ample on $Y$. Hence $a=0$ as wanted. 

Hence, for the proof of iii') it suffices to show that $\xi ^{k-r-1-q}.K_Y^q=0$. We consider two cases:

Case 1: $q<k-dim (V)$. Pushing the equation $\pi ^*(\xi ^{k-r-1-q}).(\pi ^*(c_1(Y))-(s-1)E)^q=0$ by the map $\pi$, we then find that $\xi ^{k-r-1-q}.K_Y^q=0$ as wanted. 

Case 2: $q\geq k-dim(V)$. In this case we first pushforward the equation $\pi ^*(\xi ^{k-r-1-q}).(\pi ^*(c_1(Y))-(s-1)E)^q.E=0$ by the map $\pi$ and find that
\begin{eqnarray*}
\iota _*(\sum _{q\geq i\geq k-dim (V)}(\xi |_V)^{k-r-1-q}.(c_1(Y)|_V)^{q-i}(-1)^{i-1}(\pi _E)_*(E^{i-1}))=0.
\end{eqnarray*}
As argued above, each term insided the $\iota _*$ on the LHS of the above equation is psef, therefore each of them must be zero. In particular, the term with $i=k-dim(V)$, which is $(\xi |_V)^{k-r-1-q}.(c_1(Y)|_V)^{q-k+dim (V)}$, must be zero. Since $c_1(Y)|_V$ is ample by assumption, and since $\xi |_V$ is either zero or ample on $V$, it then follows that $\xi |_V=0$. Therefore 
\begin{eqnarray*} 
0=\pi ^*(\xi ^{k-r-1-q}).(\pi ^*(c_1(Y))-(s-1)E)^q=\pi ^*(\xi )^{k-r-1-q}.\pi ^*(c_1(Y)^q).
\end{eqnarray*}
Pushing this by the map $\pi$, we find that $\xi ^{k-r-1-q}.c_1(Y)^q=0$, as wanted.
\end{proof}
 
In the border case $dim(V)=r+1$, we can make Theorem \ref{TheoremEk} ii), ii') and iii') stronger. Part i) of the following result can be regarded as a generalization of Lemma \ref{LemmaManifoldDimension4} ii). Examples satisfying Corollary \ref{CorollaryExtremeCases} will be given at the end of this section (see in particular Examples 4 and 5). 
\begin{corollary}
Let $Y$ be a projective manifold of dimension $k$, and let $V\subset Y$ be a compact complex submanifold. Let $\pi :X\rightarrow Y$ be the blowup of $Y$ at $V$, and let $E$ be the exceptional divisor. Let $\mathcal{E}=N_{V/Y}$ be the normal vector bundle of $V$ in $Y$. Let $\iota :V\rightarrow Y$ be the inclusion, and let $\pi _E:E\rightarrow V$ be the projection.

i) Assume that $Y$ satisfies the non-vanishing condition $A(r,0)$ (correspondingly the non-vanishing conditions $NA(r,0)$, $B(r,0)$ and $NB(r,0)$), and $dim(V)=r+1$. Assume moreover that $V$ is not the only effective variety (with real coefficients) in its cohomology class, and $\iota _*(c_1(\mathcal{E}))$ is not psef. Then $X$ also satisfies the non-vanishing condition $A(r,0)$ (respectively the non-vanishing conditions $NA(r,0)$, $B(r,0)$ and $NB(r,0)$).

ii) (This is a generalization of i).)  Assume that $Y$ satisfies the non-vanishing condition $A(r,0)$ (correspondingly the non-vanishing conditions $NA(r,0)$, $B(r,0)$ and $NB(r,0)$), and $dim(V)=r+1$. Assume moreover that there is an integer number $j\geq 1$, an effective cycle $V'$ having the same cohomology class as that of $V$ such that $dim(V'\cap V)\leq dim (V)-j$, and $\iota _*(c_j(\mathcal{E}))$ is not psef. Then $X$ also satisfies the non-vanishing condition $A(r,0)$ (respectively the non-vanishing conditions $NA(r,0)$, $B(r,0)$ and $NB(r,0)$).

\label{CorollaryExtremeCases}\end{corollary}
\begin{proof}

i) Let $\zeta $ be a nef class on $X$ with $\zeta ^{k-r-1}=0$. We write $\zeta =\pi ^*(\xi )-aE$ with $a\geq 0$. To prove ii) it suffices to show that $a=0$. Let $\iota _E:E\rightarrow X$ be the inclusion. Then we have $\iota _E^*(\zeta )^{k-r-1}=0$, which is the same as 
\begin{eqnarray*}
(\pi _E^*(\xi |_V)-ae)^{k-r-1}=0.
\end{eqnarray*}

Because $dim(V)=r+1$, the defining equation for $H^*(E)$ is then
\begin{eqnarray*}
e^{k-r-1}-\pi _E^*(c_1(\mathcal{E}))e^{k-r-2}+\ldots =0.
\end{eqnarray*}
Comparing this equation with the equation $(\pi _E^*(\xi |_V)-ae)^{k-r-1}=0$, it follows that $\pi _E^*(\xi |_V)=a\pi _E^*(c_1(\mathcal{E}))$. Since the pullback maps $\pi _E^*:H^*(V)\rightarrow H^*(E)$ are injective, we obtain $\xi |_V=ac_1(\mathcal{E})$. From this, we must have $a=0$. Otherwise, pushing forward by $\iota$ we obtain $\xi .V=a\iota _*(c_1(\mathcal{E}))$. This is a contradiction, since the LHS is psef (as in the proof of Theorem \ref{TheoremEk}) while the RHS is not psef by assumption.  

ii) The proof is similar to that of i): Rescaling, we may assume that $a=1$. We now use $\xi ^j|_V=C(k-r-1,j) c_j(\mathcal{E})$, and $\iota _*(\xi ^j|_V)=\xi ^j.V=\xi ^j.V'$ is psef.
\end{proof}  
 
{\bf Remark 1.} In Theorem \ref{TheoremEk} ii), we can represent $\pi _*(E^j)$ for $j=s,s+1,\ldots ,k-r$ in terms of the Chern classes of $\mathcal{E}$. Recall that $\pi :X\rightarrow Y$ is the blowup along a submanifold $V\subset Y$, $s=codim(V)$, $\mathcal{E}=N_{V/Y}$ the normal vector bundle of $V$ in $Y$, $E=\mathbb{P}(\mathcal{E})$ the exceptional divisor of $\pi$ with $\pi _E:E\rightarrow V$ the projection, and $e=E|_E$. First of all, we have $(-1)^{s-1}\pi _*(E^s)=V$ by the formula at the beginning of Section 4.3 in \cite{fulton}. To compute the pushforward of other $E^j$ we use the following formula (see Proposition 3.1 and the proofs of Lemma 3.3 and Proposition 6.7 in \cite{fulton}): 
\begin{equation}
(\pi _E)_*(\sum _{j=0}^{s-1}e^{j}.\pi _E^*(x_j))=(-1)^{s-1}x_{s-1}.
\label{Equation6}\end{equation}
It is then easy to compute the pushforward of $E^j$ for $j\geq s$. Let $\iota _E:E\subset X$ be the inclusion of $E$ in $X$. Then 
\begin{eqnarray*}
\pi _*(E^{s+1})=\pi _*(\iota _E)_*(e^s)=\iota _*(\pi _E)_*(e^s).
\end{eqnarray*}
By the defining equation 
\begin{eqnarray*}
\sum _{j=0}^{s}(-1)^je^{s-j}.\pi _E^*(c_j(\mathcal{E}))=0,
\end{eqnarray*}
we find that 
\begin{eqnarray*}
e^s=e^{s-1}\pi _E^*(c_1(\mathcal{E}))-e^{s-2}\pi _E^{*}(c_2(\mathcal{E}))+\ldots .
\end{eqnarray*}
Using (\ref{Equation6}), we find that $(\pi _E)_*(e^s)=(-1)^{s-1}c_1(\mathcal{E})$, and hence $\pi _*(E^{s+1})=(-1)^{s-1}\iota _*(c_1(\mathcal{E}))$. Similarly, we can compute $\pi _*(E^{s+2})$: We have 
\begin{eqnarray*} 
\pi _*(E^{s+2})=\pi _*(\iota _E)_*(e^{s+1})=\iota _*(\pi _E)_*(e^{s+1}).
\end{eqnarray*}
Now we have
\begin{eqnarray*}
e^{s+1}&=&e^s\pi _E^*(c_1(\mathcal{E}))-e^{s-1}\pi _E^*(c_2(\mathcal{E}))+\ldots \\
&=&(e^{s-1}\pi _E^*(c_1(\mathcal{E}))-e^{s-2}\pi _E^*(c_2(\mathcal{E}))+\ldots )\pi _E^*(c_1(\mathcal{E}))-e^{s-1}\pi _E^*(c_2(\mathcal{E}))+\ldots \\
&=&e^{s-1}\pi _E^*(c_1(\mathcal{E})^2-c_2(\mathcal{E}))+\ldots ,
\end{eqnarray*}
hence $(\pi _E)_*(e^{s+1})=(-1)^{s-1}[c_1(\mathcal{E})^2-c_2(\mathcal{E})]$. Therefore 
\begin{eqnarray*}
\pi _*(E^{s+2})=(-1)^{s-1}\iota _*(c_1(\mathcal{E})^2-c_2(\mathcal{E})).
\end{eqnarray*}
Similarly we can compute the pushforward of other $E^j$'s in terms of Chern classes of $\mathcal{E}$.

{\bf Example 1.} We now give a construction to provide many examples when the conditions of Theorem \ref{TheoremEk} can be easily checked. 

Assume that $\pi _1:Y_1\rightarrow Y$ is a blowup along a smooth submanifold $W_1\subset Y$ of dimension $d_1\geq 1$, and let $E_1$ be the exceptional divisor. Let $V\sim \mathbb{P}^{k-d_1-1}\subset E_1$ be a fiber of the restriction $\pi _1:E_1\rightarrow V_1$. Let $\pi :X\rightarrow Y_1$ be the blowup of $Y_1$ at $V$. If $Y_1$ satisfies the non-vanishing condition $A(r,q)$ (correspondingly $NA(r,q)$, $B(r,q)$ and $NB(r,q)$) then $X$ satisfies the non-vanishing condition $A(r,q)$ (respectively $NA(r,q)$, $B(r,q)$ and $NB(r,q)$). 

First, if $dim(V)\leq r$ then we can apply Theorem \ref{TheoremEk} i). Hence we can assume that $dim(V)\geq r+1$, and will show that Theorem \ref{TheoremEk} ii), ii') and iii') apply. 

In fact, let $\mathcal{E}=N_{V/Y_1}$ be the normal vector bundle of $V$ in $Y_1$.   Then we have the following SES of vector bundles over $V$ (see Appendix B.7.4 in \cite{fulton})
\begin{eqnarray*}
0\rightarrow N_{V/E_1}\rightarrow N_{V/Y_1}\rightarrow N_{E_1/Y_1}|_{V}\rightarrow 0.
\end{eqnarray*}
Therefore $c(N_{V/Y_1})=c(N_{V/E_1})c(N_{E_1/Y_1}|_V)$. Because $E_1$ is the projectivization of a vector bundle over $W_1$ and $V$ is a fiber of $E_1\rightarrow W_1$, it follows that the normal vector bundle $N_{V/E_1}$ is trivial. Hence $c(N_{V/E_1})=1$. If $h$ is the class of a hyperplane on $V$ then we have
\begin{eqnarray*}
N_{E_1/Y_1}|_V=(E_1|_{E_1})|_{V}=-h,
\end{eqnarray*}
since $E_1|_{E_1}$ is the tautological bundle. Therefore $c(N_{E_1/Y_1}|_V)=1-h$, and $c(N_{V/Y_1})=1-h$. Hence $c_0(N_{V/Y_1})=1$, $c_1(N_{V/Y_1})=-h$, and the other Chern classes are zero. Thus, if $E$ is the exceptional divisor of the blowup $\pi :X\rightarrow Y_1$ and $\pi _E:E\rightarrow V$ is the projection then the defining equation for $H^*(E)$ over $H^*(V)$ is 
\begin{eqnarray*}
e^s=\pi _E^*(c_1(\mathcal{E}))e^{s-1}. 
\end{eqnarray*} 
Therefore $$(-1)^{s+j-1}\pi _*(E^{s+j})=(-1)^{j}\iota _*(c_1(\mathcal{E})^j)=(-1)^j\iota _*((-h)^j)=\iota _*(h^j),$$ 
are all strictly effective, for $j=0,\ldots ,k-s=dim(V)$. Since in Theorem \ref{TheoremEk} ii) we assumed that $dim (V)\geq r+1$ we see that 
\begin{eqnarray*}
(-1)^{k-r-1}\pi _*(E^{k-r})=(-1)^{k-r-1}\pi _*(E^{s+k-r-s})=\iota _*(h^{dim (V)-r})
\end{eqnarray*}
is strictly effective. Also, all of $\iota _*(h^j)$ can be represented by linear subspaces of the other fibers of $\pi _{E_1}$ disjoint from $V$. Hence if $q=0$ we see that all the assumptions of Theorem \ref{TheoremEk} ii) are satisfied. In this case $V$ has Picard number $1$, and we can apply Theorem \ref{TheoremEk} ii'). We now check that $c_1(Y_1)|_V$ is ample in order to be able to apply Theorem iii') when $q>0$. In fact, from the SES of vector bundles on $V$
\begin{eqnarray*}
0\rightarrow T_V\rightarrow T_{Y_1}|_V\rightarrow \mathcal{E}\rightarrow 0,
\end{eqnarray*}
we have $c_1(Y_1)|_V=c_1(V)+c_1(\mathcal{E})$. As computed above $c_1(\mathcal{E})=-h$, and because $V$ is a projective space we find $c_1(V)=(dim(V)+1)h$. Therefore $c_1(Y_1)|_V=dim(V)h$ is ample because $dim(V)>0$. 

{\bf Example 2.} In the situation of Example 1, we can also blowup a hyperplane $W$ of $V$ to produce an example satisfying the assumptions of Theorem \ref{TheoremEk} ii) if either $dim (W)=dim (V)-1\leq r$ or the number $dim(W)-r$ is even. More precisely, assume that $\pi _1:Y_1\rightarrow Y$ is a blowup along a smooth submanifold $W_1\subset Y$ of dimension $d_1\geq 1$, and let $E_1$ be the exceptional divisor. Let $V\sim \mathbb{P}^{k-d_1-1}\subset E_1$ be a fiber of the restriction $\pi _1:E_1\rightarrow V_1$, and let $W\subset V$ be a hyperplane. Let $\pi :X\rightarrow Y_1$ be the blowup of $Y_1$ at $W$. If $Y_1$ satisfies the non-vanishing condition $A(r,q)$ (correspondingly $NA(r,q)$, $B(r,q)$ and $NB(r,q)$)
and either 

i) $dim(W)\leq r$

or 

ii) $dim(W)-r$ is even,

or 

iii) $r\geq 1$,

then $X$ satisfies the non-vanishing condition $A(r,q)$ (respectively $NA(r,q)$, $B(r,q)$ and $NB(r,q)$).  

In fact, in case i) we can apply Theorem \ref{TheoremEk} i). 

We consider case ii). The assumption in case ii) is the same as $(k-r)-s$ is even, where $s=k-dim(W)$ is the codimension of $W$. Let $E$ be the exceptional divisor of the blowup, and let $\pi _E:E\rightarrow W$ be the projection. Let $h$ be the hyperplane class in $W$. Then compute as in Example 1 we see that $$c(N_{W/Y_1})=c(N_{W/V}).c(N_{V/E_1}|_{W}).c(N_{E_1/Y_1}|_W)=(1+h)(1-h)=1-h^2.$$ 
Hence $c_0(N_{V/Y_1})=1$, $c_2(N_{V/Y_1})=-h^2$, and other Chern classes are zero. Thus the defining equation for $H^*(E)$ is 
\begin{eqnarray*}
e^s=-\pi _E^*(c_2(N_{V/Y_1}))e^{s-2}=\pi _E^*(h^2)e^{s-2}. 
\end{eqnarray*}
Then as in Example 1 we find $(-1)^{s+j-1}\pi _*(E^{s+j})=0$ if $j$ is odd, and $(-1)^{s+j-1}\pi _*(E^{s+j})=\iota _*(h^j)$ if $j$ is even. Hence all of them are effective, and if $j$ are even then they are strictly effective. Since we assume that $(k-r)-s$ is even, it follows that the term $(-1)^{k-r-1}\pi _*(E^{k-r})$ is strictly effective. Also, all of these classes can be represented by linear subspaces of the other fibers disjoint from $V$ and hence from $W$. Hence all the assumptions of Theorem \ref{TheoremEk} ii) are satisfied. In this case $W$ has Picard number $1$, and we can apply Theorem \ref{TheoremEk} ii'). We now check that $c_1(Y_1)|_W$ is ample so that Theorem \ref{TheoremEk} iii') also applies. As computed in Example 1, we have $c_1(Y_1)|_W=c_1(W)+c_1(\mathcal{E})$. Here we computed above that $c_1(\mathcal{E})=0$, and again have $c_1(W)=(dim(W)+1)h$. Hence    
$c_1(Y_1)|_W=(dim(W)+1)h$ is ample, as wanted.

In case iii) we apply the same argument as in ii) to the equation $\zeta ^{k-r-1-q}.K_X^q.E^2=0$, instead of the equation $\zeta ^{k-r-1-q}.K_X^q.E=0$.

{\bf Example 3.} We give a specific application of Examples 1 and 2. Let $X_0$ be a projective manifold of even dimension $k=2l$ satisfying the non-vanishing condition $B(l-1,q)$ (or $A(l-1,q)$, $NA(l-1,q)$ and $NB(l-1,q)$). For example, we can take $X_0=$ a projective hyper-K\"ahler manifold or a manifold with Picard number $1$. Let $1\leq j\leq l-1$ and $V_0\subset X_0$ be a manifold of dimension $j$. We let $X_1\rightarrow X_0$ be the blowup of $X_0$ along $V_0$. By Theorem \ref{TheoremEk} ii), we know that $X_1$ also satisfies the non-vanishing condition $B(l-1,q)$. Let $E_0\subset X_1$ be the exceptional divisor, then a fiber $V$ of $E_0\rightarrow V_0$ has dimension $l\leq k-1-j\leq k-2$. By Example 1, we see that if $X\rightarrow X_1$ is the blowup of $X_1$ along $V$, then $X$ also satisfies the non-vanishing condition $B(l-1,q)$. Hence by Theorem \ref{TheoremHyperKahler}, if $f\in Aut(X)$ is an automorphisms then either $h_{top}(f)=0$ or $f$ is cohomologically hyperbolic.        

{\bf Example 4.} We now give examples satisfying Corollary \ref{CorollaryExtremeCases} i) and ii). This example allows blowing up higher codimension submanifolds in Examples 1 and 2. Assume that $\pi _1:Y_1\rightarrow Y$ is a blowup along a smooth submanifold $W_1\subset Y$ of dimension $d_1\geq 1$, and let $E_1$ be the exceptional divisor. Let $V\sim \mathbb{P}^{k-d_1-1}\subset E_1$ be a fiber of the restriction $\pi _1:E_1\rightarrow V_1$, and let $W\subset V$ be a complex submanifold of dimension $r+1$. Let $\pi :X\rightarrow Y_1$ be the blowup at $W$.  

i) Assume that $2r+1\geq dim(V)\geq r+1$. If $Y_1$ satisfies the non-vanishing condition $A(r,0)$ (correspondingly $NA(r,0)$, $B(r,0)$ and $NB(r,0)$) then $X$ satisfies the non-vanishing condition $A(r,0)$ (respectively $NA(r,0)$, $B(r,0)$ and $NB(r,0)$). 

ii) Assume that $[(dim (V)-r)/2]\leq r$, here $[x]$ is the largest integer less than or $x$. Assume moreover that $W\subset V$ is a linear subspace. If $Y_1$ satisfies the non-vanishing condition $A(r,q)$ (correspondingly $NA(r,q)$, $B(r,q)$ and $NB(r,q)$) then $X$ satisfies the non-vanishing condition $A(r,q)$ (respectively $NA(r,q)$, $B(r,q)$ and $NB(r,q)$).

Proof:

i) Let $H$ be the hyperplane class in $V$, and $h=H|_W$. As in the proof of Example 2) we have 
\begin{eqnarray*}
c(N_{W/Y_1})&=&c(N_{W/V})c(N_{V/E_1}|_W)c(N_{E_1/Y_1}|_W)\\
&=&c(N_{W/V})(1-h).
\end{eqnarray*}
Let $t=dim (V)-dim (W)$ be the codimension of $W$ in $V$. Then 
\begin{eqnarray*}
c_{t+1}(N_{W/Y_1})=-hc_{t}(N_{W/V})=-deg(W)h^{t+1}, 
\end{eqnarray*}
the last equality is a consequence of the self-intersection formula (see page 103 in  \cite{fulton}). The assumption that $dim (V)\leq 2r+1$ implies that $c_{t+1}(N_{W/Y_1})$ is negative and is non-zero. Hence if $\iota :W\rightarrow Y_1$ is the inclusion map then $\iota _*(c_{t+1}(N_{W/Y_1}))$ is not psef. We can choose a submanifold $W'$ of another fiber disjoint from $V$ such that $W'$ has the same cohomology class as that of $W$. Hence Corollary \ref{CorollaryExtremeCases} ii) can be applied to complete the proof of i). 

ii) If $dim(V)=r+1$ then $W=V$ and we can apply Example 1. Therefore we need to consider only the case when $dim(V)>r+1$. Define $t=dim (V)-dim (W)=dim (V)-r-1>0$ and $s=[(t+1)/2]=[(dim(V)-r)/2]$. Since $W$ is a complete intersection of hyperplanes, as computed above we find that 
\begin{eqnarray*} 
c(N_{W/Y_1})=(1-h)(1+h)^t=\sum _{j=0}^{t+1}(C(t,j)-C(t,j-1))h^{j}.
\end{eqnarray*}  
By the properties of binomial numbers, $C(t,s)>C(t,s+1)$. This, together with the assumption that $s\leq r$ (hence $s+1\leq r+1$) implies that $c_{s+1}(N_{W/Y_1})=(C(t,s+1)-C(t,s))h^{s+1}$ is negative and is non-zero.  

Also, $W$ has Picard number $1$. We computed above that $c_1(N_{W/Y_1})=c_1(N_{W/V})-h$. Since $V$ is a projective space, we have $c_1(V)=(dim(W)+1)H$. Therefore 
\begin{eqnarray*}
c_1(Y_1)|_W&=&c_1(W)+c_1(N_{W/Y_1})=c_1(W)+c_1(N_{W/V})-h\\
&=&c_1(V)|_W-h=(dim (V)+1)H|_W-h=dim(V)h
\end{eqnarray*}
is ample. Thus we can apply the proofs of i) and Example 1. 

{\bf Example 5.} This example is to show that except for the requirement that $V$ is not the only effective cycle in its cohomology class, the condition in Corollary \ref{CorollaryExtremeCases} is always satisfied if we blowup enough points in generic positions. More precisely, let $Y$ be a complex projective manifold of dimension $k$ and let $V\subset Y$ be a proper compact complex submanifold of dimension $\geq 1$. Let $x_1,\ldots ,x_n,\ldots $ be a sequence of distinct points in $V$ such that $\bigcup _{j=1}^{\infty}x_j$ is not contained in any subvariety of dimension $dim(V)-1$. For any  $t$ , let $\pi _t:X_t\rightarrow Y$ be the blowup at $x_1,\ldots ,x_t$. Let $\widetilde{V}_t$ be the strict transform of $V$ in $X$, $\iota _t:\widetilde{V}_t\rightarrow X$ the inclusion map, and $\mathcal{E}_t=N_{\widetilde{V}_t/X}$ the normal vector bundle. If the number $t$ is large enough, then $(\iota _t)_*(c_1(\mathcal{E}_t))$ is not psef. 

Proof: From the SES of vector bundles on $\widetilde{V}_t$:
\begin{eqnarray*}
0\rightarrow T_{\widetilde{V}_t}\rightarrow T_{X_t}|_{\widetilde{V}_t}\rightarrow \mathcal{E}_t\rightarrow 0,
\end{eqnarray*}  
we have $c_1(\mathcal{E}_t)=\iota _t^*(c_1(X_t))-c_1(\widetilde{V}_t)$. Let $E_j=\pi ^{-1}(x_j)$ be the exceptional divisor over the point $x_j$. Then 
\begin{eqnarray*}
c_1(X_t)=\pi _t^*(c_1(Y))-(k-1)\sum _{j=1}^tE_j.
\end{eqnarray*}
If $p_t=\pi _t|_{\widetilde{V}_t}:\widetilde{V}_t\rightarrow V$ is the restriction of $\pi _t$ to $\widetilde{V}_t$, then $p_t$ is the blowup of $V$ at $x_1,\ldots ,x_t$ (see Example 7.17 in the book Harris \cite{harris}). Moreover, $\alpha _j=E_j\cap \widetilde{V}_t=p_t^{-1}(x_j)$ is the exceptional divisor over $x_j$ of the map $p_t$. Hence
\begin{eqnarray*}
c_1(\widetilde{V}_t)=p_t^*(c_1(V))-(dim (V)-1)\sum _{j=1}^t\alpha _j,
\end{eqnarray*}  
and 
\begin{eqnarray*}
c_1(\mathcal{E}_t)&=&\iota _t^*\pi _t^*(c_1(Y))-p_t^*(c_1(V))-(k-dim (V))\sum _{j=1}^t\alpha _j\\
&=&p_t^*(c_1(Y)|_V)-p_t^*(c_1(V))-(k-dim (V))\sum _{j=1}^t\alpha _j.
\end{eqnarray*} 

Now we show that $(\iota _t)_*(c_1(\mathcal{E}_t))$ is not psef for $t$ large enough. 

We consider two cases:

Case 1: $V$ has dimension $1$. In this case, $c_1(Y)|_V$ and $c_1(V)$ are just numbers, and hence $c_1(\mathcal{E}_t)=c_1(Y)|_V-c_1(V)-t(k-dim(V))$ is negative when $t>c_1(Y)|_V-c_1(V)$. 

Case 2: $V$ has dimension $\geq 2$. Assume otherwise that there are large values of $t$ (as large as desired) such that the cohomology class $(\iota _t)_*(c_1(\mathcal{E}))$ can be represented by positive closed currents $S_t$ of bidimension $(dim(V)-1,dim(V)-1)$ on $X_t$. We first check that for such currents $S_t$, the positive closed current $(\pi _t)_*(S_t)$ has Lelong number $\geq 1$ at $x_1,\ldots ,x_t$. We check this for example at the point $x_1$. Using the projection from $X_t$ to the blowup of $Y$ at $x_1$, it is enough to check the claim for the case $t=1$. We then need to check that $(\pi _1)_*(S_1)$ has Lelong number at least $1$ at $x_1$. Subtracting $S_1$ from its restriction to $E_1$ if needed, we may assume that $S_1$ has no mass on $E_1$ and in cohomology $\{S_1\}=(\iota _1)_*\{p_1^*(c_1(Y)|_V)-p_1^*(c_1(V))-(k-dim (V)+a)\alpha _1\}$ where $a\geq 0$. Let $b\geq 0$ be the Lelong number of $(\pi _1)_*(S_1)$ at $x_1$. By Siu's theorem (see Siu \cite{siu}) on Lelong numbers combined with the approximation theorem for positive closed currents on compact K\"ahler manifolds of Dinh and Sibony (see Dinh-Sibony \cite{dinh-sibony1}), in cohomology $\{(\pi _1)^*(\pi _1)_*(S_1)\}=\{S_1\}+b\{\iota _*(\alpha _1)\}$. Then we can intersect with $E_1^{dim(V)-1}$ to obtain that $b=k-dim(V)+a\geq k-dim (V)\geq 1$.

We now finish the proof of Case 2. The positive closed currents $(\pi _t)_*(S_t)$ on $X$ have the same cohomology class:
\begin{eqnarray*}
(\pi _t)_*\{S_t\}&=&(\pi _t)_*(\iota _t)_*(c_1(\mathcal{E}))\\
&=&\iota _*(p_t)_*[p_t^*(c_1(Y)|_V)-p_t^*(c_1(V))-(k-dim (V))\sum _{j=1}^t\alpha _j]\\
&=&\iota _*(c_1(Y)|_V-c_1(V)).
\end{eqnarray*}
In particular, they have uniformly bounded masses, and we can extract a cluster point $S$, which is a positive closed current on $X$ of bidimension $(dim(V)-1,dim(V)-1)$. 
Since $(\pi _t)_*(S_t)$ has Lelong number at least $1$ at $x_{t_0}$ for $t\geq t_0$, it follows by the upper-semicontinuity of Lelong numbers (see e.g. Chapter 3 in \cite{demailly}), $S$ has Lelong number at least $1$ at the points $x_1,x
_2,\ldots $. But this contradicts to Siu's theorem (see \cite{siu}) that the set of points where $S$ has Lelong number at least $1$ is a subvariety of $X$ of dimension $\leq dim (V)-1$ and our assumption that $\bigcup _{j=1}^{\infty}x_j$ is not contained in a variety of dimension $dim(V)-1$. 

{\bf Example 6.} In this example, we show that the assumptions in Theorem \ref{TheoremEk} ii') (and those of Lemmas \ref{LemmaP4} and \ref{LemmaManifoldDimension4} and Corollary \ref{CorollaryExtremeCases}) can not be removed. We consider here $Y$ a complex projective manifold of dimension $3$ and $C\subset Y$ a smooth curve isomorphic to $\mathbb{P}^1$. Assume that $Y$ satisfies the non-vanishing condition $A(r,q)$ (where $r\geq 0$) (correspondingly the non-vanishing conditions $B(r,q)$, $NA(r,q)$ and $NB(r,q)$). In Corollary \ref{CorollaryExtremeCases} and Example 4 we showed that if the following two conditions are satisfied:

i) $c_1(Y).C-c_1(C)<0$, 

and 

ii) $C$ is not the only effective cycle in its cohomology class,

then $X$ also satisfies the non-vanishing condition $A(r,q)$ (respectively the non-vanishing conditions $B(r,q)$, $NA(r,q)$ and $NB(r,q)$). 

We now show that if either one of these two conditions is removed then the above result is no longer true. 

Proof: 

1) (This example was given in Section 4.1 in \cite{truong}): It was proved by McMullen (see \cite{mcmullen1}) that there are distinct points $x_1,\ldots ,x_m\in \mathbb{P}^2$ such that the blowup $\pi :S\rightarrow \mathbb{P}^2$ at these points has an automorphism of positive entropy. Consider $X=S\times \mathbb{P}^1$, then $X$ also has an automorphism of positive entropy. Thus $X$ does not satisfy the non-vanishing condition $A(0,0)$. $X$ can also be represented as the blowup of $Y=\mathbb{P}^2\times \mathbb{P}^1$ at the disjoint curves $x_1\times \mathbb{P}^1,\ldots , x_m\times \mathbb{P}^1$, each of these curves is isomorphic to $\mathbb{P}^1$. Here $Y$ satisfies the non-vanishing condition $A(0,0)$. In this example we can see that $c_1(Y)|_{x_j\times \mathbb{P}^1}-c_1(x_j\times \mathbb{P}^1)=0$, hence condition i) above is not satisfied. Since the curves $x_j\times \mathbb{P}^1$ move in a family of dimension $2$ of curves, the condition ii) above is satisfied. 

2) Assume now that the following claim is true:

Claim 1: For any $Y$ of dimension $3$ satisfying the non-vanishing condition $A(0,0)$ and $C\subset Y$ a smooth curve isomorphic to $\mathbb{P}^1$ such that $c_1(Y)|_C-c_1(C)<0$ , the manifold $X$ always satisfies the non-vanishing condition $A(0,0)$. Then we will arrive at a contradiction. First, we show the following:

Claim 2: Assume that Claim 1 is true. Then for any $Y$ of dimension $3$ satisfying the non-vanishing condition $A(0,0)$ and $C\subset Y$ a smooth curve isomorphic to $\mathbb{P}^1$, the manifold $X$ always satisfies the non-vanishing condition $A(0,0)$. (In Claim 2 we do not need the assumption $c_1(Y)|_C-c_1(C)<0$.)

Proof of Claim 2: Let $t$ be a large number and $x_1,\ldots ,x_t$ be distinct points in $C$. Let $Z_1\rightarrow Y$ be the blowup of $Y$ at $x_1,\ldots ,x_t$, and let $Z\rightarrow Z_1$ be the blowup at the strict transform $\widetilde{C}$ of $C$. Since $Y$ satisfies $A(0,0)$ and $Z_1\rightarrow Y$ is a composition of point-blowups, $Z_1$ also satisfies $A(0,0)$. In Example 5 we showed that for $t$ large then $c_1(Z_1)|_{\widetilde{C}}-c_1(\widetilde{C})<0$. By Claim 1 applied to the blowup $Z\rightarrow Z_1$, $Z$ also satisfies $A(0,0)$. 

Next we show that $X$, which is the blowup of $Y$ at the curve $C$, also satisfies the condition $A(0,0)$. Let $\tau :Z\rightarrow Y$ be the composition of the blowups $Z\rightarrow Z_1$ and $Z_1\rightarrow Y$. We first check that $\tau ^{-1}(C)$, as a subscheme of $Z$, is a hypersurface. This is easy to see on the level of sets. Now we check that the ideal of $\tau ^{-1}(C)$ is locally generated by an element. This question is local, hence we reduce to the case where $Y=\mathbb{C}^3$, $x_1=(0,0,0)$ and $C=\{x=y=0\}$. Then a local coordinate for $Z$ is given by (see e.g. Section 1 in \cite{bedford-kim3}): $\tau :Z\ni (t_0,\eta _1, \xi _2)\mapsto (t_0\eta _1\xi _2, \eta _1\xi _2,\xi _2)\in Y$. Hence $\tau ^{-1}(C)$ is generated by $t_0\eta _1\xi _2$ and $\eta _1\xi _2$, hence is generated by one element $\eta _1\xi _2$ as claimed. 

Therefore, applying the universal property of blowups (see e.g. Proposition 7.14 in Chapter 2 in Hartshorne \cite{hartshorne} and Theorem 4.1 in Chapter 4 in Fischer \cite{fischer}), there is a birational holomorphic map $\sigma :Z\rightarrow X$. Then we can finish the proof of Claim 2 as follows: Let $\zeta$ be a nef class on $X$ such that $\zeta ^2=0$. Then $\xi =\sigma ^*(\zeta )$ is a nef class on $Z$ such that $\xi ^2=0$. Since $Z$ satisfies the non-vanishing condition $A(0,0)$, it follows that $\xi \in H^{1,1}(Z,\mathbb{Q})$. Then $$\zeta =\sigma _*(\xi )\in \sigma _*(H^{1,1}(Z,\mathbb{Q}))\subset H^{1,1}(X,\mathbb{Q}).$$ 
Thus $X$ satisfies the non-vanishing condition $A(0,0)$, and Claim 2 is proved. 

Finally, we obtain a contradiction to Claim 1. Let $x_1,\ldots ,x_m\in \mathbb{P}^2$ be such that the blowup $S\rightarrow \mathbb{P}^2$ has an automorphism of positive entropy. Let $X=S\times \mathbb{P}^1$, then $X$ does not satisfy the non-vanishing condition $A(0,0)$. However, $X$ is a blowup of $Y=\mathbb{P}^2\times \mathbb{P}^1$ at curves $x_1\times \mathbb{P}^1,\ldots ,x_m\times \mathbb{P}^1$, and $Y$ satisfies the non-vanishing condition $A(0,0)$. Hence if Claim 1 were true, then by Claim 2 $X$ also satisfies the non-vanishing condition $A(0,0)$, which is impossible. Hence Claim 1 is not true.    

\section{Proofs of Theorems \ref{TheoremArqCondition}, \ref{TheoremBrqCondition}, \ref{TheoremPicardNumber1}, \ref{TheoremAmpleCanonicalDivisor} and \ref{TheoremHyperKahler}}
\label{SectionProofsOfTheorems}

Before giving the proofs of the results in Section \ref{Introduction}, we recall some facts about spectral radius of automorphisms. The readers may see e.g. \cite{dinh-sibony}, \cite{bayraktar-cantat} or \cite{truong} and references therein for more on these facts. 

Let $X$ be a compact K\"ahler manifold of dimension $k$ and let $f:X\rightarrow X$ be an automorphism. Then we define the dynamical degrees $\lambda _p(f)=$the spectral radius of the pullback map $f^*:H^{p,p}(X)\rightarrow H^{p,p}(X)$. Then $\lambda _p(f)$ are log-concave, in particular $\lambda _p(f)\leq \lambda _1(f)^p$ for all $p=0,\ldots ,k$. Also, by Poincare duality $\lambda _p(f)=\lambda _{k-p}(f^{-1})$. We can also compute dynamical degree in the following way: Let $\omega$ be a K\"ahler class and choose an arbitrary norm on $H^{p,p}(X)$. Then 
\begin{eqnarray*}
\lambda _p(f)=\lim _{n\rightarrow\infty}||(f^n)^*(\omega ^p)||^{1/n}.
\end{eqnarray*}
The cone of nef cohomology classes $H^{1,1}_{nef}(X)$ is preserved by $f^*$. Then by a result of linear algebra, there is a non-zero $\zeta \in H^{1,1}_{nef}(X)$ such that $f^*(\zeta )=\lambda _1(f)\zeta$. Note that if $\lambda _1(f)>1$ then it is irrational. Since $f^*$ preserves $H^{2}(X,\mathbb{Z})$ it follows that when $\lambda _1(f)>1$ the eigen-class $\zeta$ can not be proportional to a rational cohomology class.  
  
The algebraic analogs of the above facts are as follows. Let $X$ be a complex projective manifold of dimension $k$. Then $X$ is K\"ahler and we can define dynamical degrees as above. Moreover, $f^*$ preserves $NS(X)$, and hence also preserves $NS_{\mathbb{Q}}(X)$ and $NS_{\mathbb{R}}(X)$. Since $NS_{\mathbb{R}}(X)\subset H^{1,1}(X,\mathbb{R})$ and $NS_{\mathbb{R}}(X)$ contains ample divisors, it follows that $\lambda _1(f)=$the spectral radius of $f^*:NS_{\mathbb{R}}(X)\rightarrow NS_{\mathbb{R}}(X)$. Again, there is a non-zero nef class $\zeta\in NS_{\mathbb{R}}(X)$ such that $f^*(\zeta )=\lambda _1(f)\zeta$, and if $\lambda _1(f)>1$ then such a $\zeta$ can not be proportional to an element in $NS_{\mathbb{Q}}(X)$.    

Now we give the proofs of Theorems \ref{TheoremArqCondition}, \ref{TheoremBrqCondition}, \ref{TheoremPicardNumber1}, \ref{TheoremAmpleCanonicalDivisor} and \ref{TheoremHyperKahler}.

\begin{proof}[Proof of Theorem \ref{TheoremArqCondition}] We prove for the case of non-vanishing condition $A(r,q)$, the case of non-vanishing condition $NA(r,q)$ is similar. 

Let $f:X\rightarrow X$ be an automorphism. We will show that $\lambda_1(f)=0$. Assume otherwise, i.e. $\lambda _1(f)>1$, and we will reach a contradiction. Let $\zeta\in H^{1,1}(X)$ be a non-zero nef class such that $f^*(\zeta )=\lambda _1(f)\zeta$. Then $\zeta $ is not proportional to a rational cohomology class, hence by the non-vanishing condition $A(r,q)$ we have $\zeta ^{k-r-1-q}.K_X^q\not= 0$. In particular, $\zeta ^{k-r-1-q}\not= 0$ and $f^*(\zeta ^{k-r-1-q})=\lambda _1(f)^{k-r-1-q}\zeta ^{k-r-1-q}$. Hence we deduce $\lambda_{k-r-1-q}(f)\geq \lambda _1(f)^{k-r-1-q}$. The log-concavity of dynamical degrees implies
\begin{eqnarray*}
\lambda _j(f)=\lambda _1(f)^j
\end{eqnarray*}    
for all $0\leq j\leq k-r-1-q$. The assumption that $k-r-1-q>r+1$ and $\lambda _1(f)>1$ implies 
\begin{eqnarray*}
\lambda _{r+1}(f)=\lambda _1(f)^{r+1}<\lambda _1(f)^{k-r-1-q}. 
\end{eqnarray*} 
Because $f$ is an automorphism, we have $f^*(K_X)=K_X$. Therefore, since $\zeta ^{k-r-1-q}.K_X^q\not= 0$ and $f^*(\zeta ^{k-r-1-q}.K_X^q)=\lambda _1(f)^{k-r-1-q}\zeta ^{k-r-1-q}.K_X^q$, we deduce that 
$$\lambda _{k-r-1}(f)\geq \lambda _1(f)^{k-r-1-q}>\lambda _1(f)^{r+1}=\lambda _{r+1}(f).$$ 
Apply the same argument to the inverse map $f^{-1}$, we have $\lambda _{k-r-1}(f^{-1})>\lambda _{r+1}(f^{-1})$. Since $\lambda _{k-r-1}(f^{-1})=\lambda _{r+1}(f)$ and $\lambda _{r+1}(f^{-1})=\lambda _{k-r-1}(f)$ we obtain $\lambda _{r+1}(f)>\lambda _{k-r-1}(f)$. But the last inequality is contradict to the inequality $\lambda _{k-r-1}(f)>\lambda _{r+1}(f)$ which we obtained before, and we have a contradiction. Therefore $\lambda _1(f)=1$ as wanted. 
\end{proof}
\begin{proof}[Proof of Theorem \ref{TheoremBrqCondition}] We prove for the case of the non-vanishing condition $B(r,q)$, the case of $NB(r,q)$ is similar.

From Theorem \ref{TheoremArqCondition}, we already know that if $f\in Aut(X)$ then $\lambda _1(f)=1$. Replacing $f$ by an iterate of $f$ if needed, we can assume that all of the eigenvalues of $f^*:H^{1,1}(X)\rightarrow H^{1,1}(X)$ are $1$. To show that $Aut(X)$ has only finitely many connected components, it suffices to show that the size of the largest Jordan block of $f^*:H^{1,1}(X)\rightarrow H^{1,1}(X)$ is $1$.

We follow Steps 2-4 of the proof of Theorem 1.1 in \cite{bayraktar-cantat}. For each $p=1,\ldots ,k$, we choose a norm $||.||_p$ on $H^{p,p}(X)$, and define 
\begin{eqnarray*}
m_p(f)=\lim _{n\rightarrow \infty}\frac{\log ||(f^*)^n||_p}{\log n}. 
\end{eqnarray*} 
Then to prove Theorem \ref{TheoremBrqCondition}, it suffices to show that $m_1(f)=0$. 

First, we observe that the map $p\mapsto m_p(f)$ is concave, i.e. $m_{p-1}(f)+m_{p+1}(f)\leq 2m_p(f)$. This can be proved as proving the log-concavity of the dynamical degrees of $f$, see \cite{bayraktar-cantat} for more details. 

Now we show that $m_1(f)=0$. Otherwise, then $m_1(f)\geq 1$ and we will obtain a contradiction. If we let $\omega$ be a K\"ahler class, then the sequence $n^{-m_1(f)}(f^*)^n(\omega )$ converges to a non-zero nef class $u$. By the non-vanishing condition $B(r,q)$, it follows that $u^{k-r-1-q}.K_X^q\not= 0$. In particular, $u^{k-r-1-q}\not= 0$. Hence 
\begin{eqnarray*} 
\lim _{n\rightarrow\infty}n^{-m_1(f)(k-r-1-q)}(f^*)^n(\omega ^{k-r-1-q})=u^{k-r-1-q}\not= 0.
\end{eqnarray*}
In particular $m_{k-r-1-q}(f)\geq (k-r-1-q)m_1(f)$, and the concavity of $p\mapsto m_p(f)$ implies that $m_j(f)=jm_1(f)$ for all $0\leq j\leq k-r-1-q$. Since $k-r-1-q>r+1$ and we assume that $m_1(f)>0$ we have
\begin{eqnarray*}
m_{r+1}(f)=(r+1)m_1(f)<(k-r-1-q)m_1(f).
\end{eqnarray*}

We also have
\begin{eqnarray*} 
\lim _{n\rightarrow\infty}n^{-m_1(f)(k-r-1-q)}(f^*)^n(\omega ^{k-r-1-q}.K_X^q)=u^{k-r-1-q}.K_X^q\not= 0.
\end{eqnarray*}
Thus $m_{k-r-1}(f)\geq (k-r-1-q)m_1(f)>m_{r+1}(f)$. Apply the same argument to $f^{-1}$ we obtain $m_{k-r-1}(f^{-1})>m_{r+1}(f^{-1})$. But by Poincare duality, we have $m_{k-r-1}(f^{-1})=m_{r+1}(f)$ and $m_{r+1}(f^{-1})=m_{k-r-1}(f)$, and obtain a contradiction. Therefore we must have $m_1(f)=0$, as wanted.
\end{proof}
 
\begin{proof}[Proof of Theorem \ref{TheoremPicardNumber1}]

Since $X_0$ has Picard number $1$, if $\zeta \in NS_{\mathbb{R}}(X_0)$ is nef and non-zero then it is ample, and hence $\zeta ^k\not= 0$. Hence $X_0$ satisfies the non-vanishing condition $B(r,0)$ for any $r\geq 0$.

1) Since the submanifolds $V_1,\ldots ,V_m$ are pairwise disjoint, arguing as in proof of Lemma \ref{LemmaP4}, it suffices to consider the case where $\pi :X_1\rightarrow X_0$ is a single blowup along a smooth manifold $V$ of a fixed dimension $dim(V)$. We now need to show that for a generic choice of $V$, then $X_1$ satisfies the non-vanishing condition $NB(1,0)$, i.e. if $\zeta \in Nef(X)=H^{1,1}(X)\cap NS_{\mathbb{R}}(X)$ is non-zero then $\zeta ^{k-2}\not= 0$. Let $H\in NS_{\mathbb{R}}(X)$ is an ample divisor, normalized so that $H^{dim (V)}|V=1$. Then we may assume that $\zeta =\pi ^*(H)+bE$ for some real number $b$ (we do not need to use the fact the $b$ is non-positive here). Note that $b$ can not be zero, because then $\pi ^*(H^{k-2})=0$, which is absurd. 

We consider two cases:

Case 1: $dim(V)\leq 1$. Pushing forward $\zeta ^{k-2}=0$ by $\pi $ we find that $H^{k-2}=0$ which is absurd since $H$ is ample. 

Case 2: $dim(V)\geq 2$ and $V$ is a complete intersection of smooth hypersurfaces $D_1,\ldots ,D_t$ here $t=$ codimension of $D$ is fixed, where $D_j=d_jH$ in $NS_{\mathbb{R}}(X)$ for positive rational numbers $d_j$. We now show that for a generic choice of $d_1,\ldots ,d_t$ then there is no $b\in \mathbb{R}$ such that $(\pi ^*(H)+bE)^{k-2}=0$. In fact, if $(\pi ^*(H)+bE)^{k-2}=0$ then both $(\pi ^*(H)+bE)^{k-2}.E^2=0$ and $(\pi ^*(H)+bE)^{k-2}.\pi ^*(H)E=0$. 

We expand these two  equations explicitly in several first terms
\begin{eqnarray*}
0&=&(\pi ^*(H)+bE)^{k-2}.E^2=b^{k-2}E^k+C(1,k-2)b^{k-3}E^{k-1}.\pi ^*(H)\\
&&+C(2,k-2)b^{k-4}E^{k-2}.\pi ^*(H^2)+C(3,k-2)b^{k-5}E^{k-3}.\pi ^*(H^3)\\
&&+C(4,k-2)b^{k-6}E^{k-4}.\pi ^*(H^4)\ldots \\ 
0&=&(\pi ^*(H)+bE)^{k-2}.E.\pi ^*(H)=b^{k-2}E^{k-1}.\pi ^*(H)+C(1,k-2)b^{k-3}E^{k-2}.\pi ^*(H^2)\\
&&+C(2,k-2)b^{k-4}E^{k-3}.\pi ^*(H^3)+C(3,k-2)b^{k-5}E^{k-4}.\pi ^*(H^4)+\ldots 
\end{eqnarray*}

Observation 1: Note also that in the polynomial $(\pi ^*(H)+bE)^{k-2}.E^2$, the coefficients of $b^{k-2-j}$ (where $j>dim(V)$) are zeros (because then $\pi _*(E^{k-j})=0$). Similarly, in the polynomial $(\pi ^*(H)+bE)^{k-2}.\pi ^*(H)E$ the coefficients of $b^{k-2-j}$ (where $j>dim (V)-1$) are zeros. 

Using Observation 1 and that $b\not= 0$, we define two polynomials 
\begin{eqnarray*}
 f(b)&:=&\frac{1}{b^{k-2-dim(V)}}(\pi ^*(H)+bE)^{k-2}.E^2\\
 g(b)&:=&\frac{1}{b^{k-1-dim(V)}}(\pi ^*(H)+bE)^{k-2}.E.\pi ^*(H).
\end{eqnarray*}

We deduce that any value $b$ for which $(\pi ^*(H)+bE)^{k-2}=0$ must be a common zero of $f(b)$ and $g(b)$. Let $\mathcal{E}=N_{V/X_0}$ is the normal vector bundle of $V$ in $X_0$. Then as in Remark 1 at the end of Section \ref{SectionBlowupsAndNonVanishingConditions}, we see that the coefficients of $f(b)$ and $g(b)$ can be described in terms of Chern classes $c_j(\mathcal{E})$. Since $V$ is the complete intersection of $D_1,\ldots ,D_t$ and $D_j=d_jH$ in $NS_{\mathbb{R}}(X)$, it follows from Example 3.2.12 that $c(\mathcal{E})=(1+d_1H)\ldots (1+d_tH)$. From this we see that $$c_j(\mathcal{E}).H^{dim (V)-j}|_V=s_j(d_1,\ldots ,d_t)H^{dim(V)}|_V,$$        
where $s_j(d_1,\ldots ,d_t)$ is the $j$-th elementary symmetric function of $d_1,\ldots ,d_t$. Therefore the coefficients of $f(b)$ and $g(b)$ are polynomials in variables $d_1,\ldots ,d_t$.

Now these two polynomials $f(b)$ and $g(b)$ has at least one common solution if and only if their resultant $R(f,g)=0$. Since the coefficients of $f(b)$ and $g(b)$ are polynomials in variables $d_1,\ldots ,d_t$, this resultant $R(f,g)$ is also a polynomial in variables $d_1,\ldots ,d_t$. Hence either $R(f,g)$ is zero identically or it is non-zero for a generic choice of $d_1,\ldots ,d_t$. Thus if we can show that $R(f,g)\not= 0$ for a special choice of rational numbers $d_1,\ldots ,d_t$ then 1) is proved. Here we only use that $f(b)$ and $g(b)$ are polynomials of $b$ whose coefficients are fixed polynomials in variables $d_1,\ldots ,d_t$, and not the fact that they were constructed from some submanifolds $V$ of $X_0$. Hence we do not need to choose $d_1,\ldots ,d_t$ to be positive numbers, hence the choice may not correspond to any actual submanifold $V$. The special values we choose now is $d_1=1$ and $d_2=\ldots =d_t=0$. We will show that for this choice then $R(f,g)\not= 0$, and it is the same as showing the two polynomials $f(b)$ and $g(b)$ has no common zero. As stated before, this case does not correspond to any actual $V$, but it does not affect the argument below, $f(b)$ and $g(b)$ are formally constructed from the Chern classes of $V$ not from $V$ itself. Hence we assume that $d_1=1$ and $d_2=\ldots =d_t=0$ correspond to a (virtual) manifold $V$. This means that $V$ is a (virtual) manifold of codimension $t$ having $c_0(\mathcal{E})=1,$ $c_1(\mathcal{E})=H|_V$ and other Chern classes are zeros, and we ignore the fact that such a $V$ can not be a complete intersection when $t>1$ or can exist at all.    

(This choice of $d_1=1$ and $d_2=\ldots =d_t=0$ can actually be made rigorous as follows. Assume that for any $V$ which is a complete intersection, then the two polynomials $f_V(b)$ and $g_V(b)$ has a common solution. We choose in particular $d_1=d$, and $d_2=\ldots =d_t=1$, here $d$ can be as large as we desire. The idea is to take the limit when $d$ goes to $\infty$. Note that when $d$ is large enough then $\pi _*(E^{t+j})$ is approximately $c_1(\mathcal{E})^j$ and hence is approximately $d^jH^j$, as can be seen from the computations in Example 1 in Section \ref{SectionBlowupsAndNonVanishingConditions}. If we rescale $\widetilde{f}_V(b)=d^{dim(V)-2}f_V(b/d)$ and $\widetilde{g}_V(b)=d^{dim(V)-1}g_V(b/d)$, then we see that the polynomials $\widetilde{f}_V(b)$ and $\widetilde{g}_V(b)$ have bounded coefficients, of bounded degrees, and have top coefficients bounded away from zero. Therefore, since they have a common solution for any choice of $d_1=d$, the same is true for their limits. Their limits are exactly the polynomials $f(b)$ and $g(b)$ corresponding with the choice of $d_1=1$ and $d_2=\ldots =d_t=0$.)

Then as in Example 2 at the end of Section \ref{SectionBlowupsAndNonVanishingConditions}, we have the defining equation for $H^*(E)$ is $e^t=\pi _E^*(H|_V)e^{t-1}$. Then we can use the computations in Example 2 to show that
\begin{eqnarray*}
 f(b)&=&\frac{1}{b^{k-2-dim(V)}}\times\\
 &&[b^{k-2}+C(1,k-2)b^{k-3}+C(2,k-2)b^{k-4}+C(3,k-2)b^{k-5}+C(4,k-2)b^{k-6}\ldots ],\\
 g(b)&=&\frac{1}{b^{k-1-dim(V)}}\times\\
 && [b^{k-2}+C(1,k-2)b^{k-3}+C(2,k-2)b^{k-4}+C(3,k-2)b^{k-5}+\ldots ].
\end{eqnarray*}
Here we use the convenience (see the definition of $f$ and $g$) that the coefficients for $b^{k-2-j}$ in the bracket for $f$ are zero when $j>dim (V)$, and the coefficients for $b^{k-1-j}$ in the bracket for $g$ are zero when $j>dim (V)$. Now it is easy to arrive at the proof of 1). We present in the below the proof for the cases $dim(V)=0,1,2,3$, the proofs for other cases are similar and hence are omitted. 

Case $dim(V)=0$: In this case, the equations $f(b)=0$ and $g(b)=0$ becomes
\begin{eqnarray*}
0&=&f(b):=1\\ 
0&=&g(b):=0,
\end{eqnarray*}
and this system has no solution.   

Case $dim(V)=1$: In this case, the equations $f(b)=0$ and $g(b)=0$ becomes 
\begin{eqnarray*}
0&=&f(b):=b+C(1,k-2)\\
0&=&g(b):=1,
\end{eqnarray*}
and this system has no solution. 

Case $dim(V)=2$: In this case, the equations $f(b)=0$ and $g(b)=0$ becomes 
\begin{eqnarray*}
0&=&f(b):=b^2+C(1,k-2)b+C(2,k-2)\\
0&=&g(b):=b+C(1,k-2).
\end{eqnarray*}
Since $f(b)=bg(b)+C(2,k-2)$ it follows that this system has no solution. 

Case $dim(V)=3$: In this case, the equations $f(b)=0$ and $g(b)=0$ becomes 
\begin{eqnarray*}
0&=&f(b):=b^3+C(1,k-2)b^2+C(2,k-2)b+C(3,k-2)\\
0&=&g(b):=b^2+C(1,k-2)b+C(2,k-2).
\end{eqnarray*}
Since $f(b)=bg(b)+C(3,k-2)$ it follows that this system has no solution. 

2) and 3) follows from 1) and Theorems \ref{TheoremBrqCondition} and \ref{TheoremEk}. 

\end{proof} 
 
\begin{proof}[Proof of Theorem \ref{TheoremAmpleCanonicalDivisor}]

1) Let $r=k-l-3$ and $q=2l+3-k$. By assumption we have $r,q\geq 0$. Observe also that $k-r-1-q=k-l-1=r+2>r+1$. By assumption, if $\zeta \in NS_{\mathbb{R}}(X_0)$ is nef and non-zero then $\zeta ^{k-r-1-q}=\zeta ^{k-l-1}\not= 0$. Because $K_{X_0}$ is anti-ample, we have $\zeta ^{k-r-1-q}.K_{X_0}^q$ is also non-zero. Hence we can apply Theorems \ref{TheoremBrqCondition} and \ref{TheoremEk} i).

2) 

i) We observe that $K_{X_0}$ is anti-ample. Hence applying 1), it suffices to show that $X_0$ satisfies the non-vanishing condition $A(k-l-1,0)$. Let $\zeta \in H^{1,1}_{nef}(X_0)$ be non-zero and such that $\zeta ^l=0$, we will show that $\zeta $ is proportional to a rational cohomology class. Since $l\leq k_1+k_2$ this would follow if we can show the claim for $l=k_1+k_2$. To this end, first we observe that $\zeta =\sum _{j=1}^m \pi _j^*(\zeta _j)$ where $\zeta _j\in H^{1,1}(\mathbb{P}^{k_j})$ are nef, and $\pi _j:X_0\rightarrow X_{0,j}$ are the projections. Since
\begin{eqnarray*}
0=\zeta ^{k_1+k_2}\geq \sum _{i<j}\pi _i^*(\zeta _i^{k_1}).\pi _j^*(\zeta _j^{k_2}),
\end{eqnarray*}
and $k_i\geq k_1$ and $k_j\geq k_2$ in the above, we have
\begin{eqnarray*}
\pi _i^*(\zeta _i^{k_1}).\pi _j^*(\zeta _j^{k_2})=0
\end{eqnarray*}
for any $i<j$. It follows that there is at most one index $i_0$ such that $\zeta _{i_0}\not= 0$. Then $\zeta =\pi _{i_0}^*(\zeta _{i_0})\in \pi _{i_0}^*(H^{1,1}(\mathbb{P}^{k_{i_0}}))$ is proportional to a rational cohomology class.  

ii) Similar to the proof of i).
\end{proof}
\begin{proof}[Proof of Theorem \ref{TheoremHyperKahler}] For the first part of the theorem, it suffices to show that $X_0$ satisfies the non-vanishing condition $B(l-1,0)$. The latter is a result of Verbitsky, see \cite{verbitsky}.

We end the proof showing $\lambda _{2l-p}(f)=\lambda _p(f)=\lambda _1(f)^p$ for any $0\leq p\leq l$. In \cite{oguiso1}, this was proved by first showing that $\lambda _1(f)=\lambda _1(f^{-1})$ using special properties of automorphisms on hyper-K\"ahler manifolds. However, this can be proved directly for any compact K\"ahler manifold of even dimension $k=2l$ satisfying the condition $A(l-1,0)$ as follows. It suffices to show the claim for $\lambda _1(f)>1$. In this case, the eigenvector $\zeta$ for $\lambda _1(f)$ is not rational, hence $\zeta ^{l}\not= 0$. Therefore $\lambda _p(f)=\lambda _1(f)^p$ for $0\leq p\leq l$, by the log-concavity of dynamical degrees. Similarly, for $0\leq p\leq l$, we have $\lambda _{2l-p}(f)=\lambda _p(f^{-1})=\lambda _1(f^{-1})^p$. Because $l=2l-l$, we have $\lambda _l(f)=\lambda _{l}(f^{-1})$ by Poincare duality. Therefore, $$\lambda _1(f^{-1})^l=\lambda _l(f^{-1})=\lambda _l(f)=\lambda _1(f)^l,$$
which implies $\lambda _1(f^{-1})=\lambda _1(f)$, and the proof is completed. 

\end{proof}

\end{document}